\documentclass[a4paper]{article}
\usepackage[all]{xy}
\usepackage{amsmath, amssymb, amsthm}
\usepackage{latexsym, color}
\usepackage[dvipdfmx, hiresbb]{graphicx}
\usepackage{float}

\newtheorem{thm}{Theorem}[section]
\newtheorem{prop}[thm]{Proposition}
\newtheorem{lem}[thm]{Lemma}
\newtheorem{defn}[thm]{Definition}

\newtheorem{cor}[thm]{Corollary}
\newtheorem{rem}[thm]{Remark}

\title{Twisted Lubin--Tate Big Witt Vectors and Fleck--Sun--Wan Congruences}
\author{
Yutaro Matsuno\\
Department of Mathematics, Waseda University\\
Tokyo, Japan\\
\texttt{nasuuy-0115@ruri.waseda.jp}
}
\date{}

\newcommand{\ds}{\displaystyle}
\newcommand{\p}{\mathfrak{p}}

\newcommand{\ded}{\mathfrak{o}}

\newcommand{\id}{\operatorname{id}}

\begin{document}

\maketitle

\begin{abstract}
We construct a big Witt theory associated with coefficient-multiplicative normalized $\sigma$-twisted Lubin--Tate formal groups.  The resulting $F$-big Witt ring extends Hazewinkel's ramified $q$-typical Witt theory and recovers the usual big Witt ring in the multiplicative case. We also introduce Fleck operators defined from iterated Coleman traces and prove Lubin--Tate analogues of the Fleck--Sun--Wan congruences for both positive and negative powers.  By relating the Lubin--Tate Coleman theory to the $F$-big Witt theory, we obtain a Coleman $F$-norm whose integrality yields congruences among the corresponding Fleck coefficients.
\end{abstract}

\section{Introduction}

Lubin and Tate\cite{LubinTate} developed the theory of formal $\ded_k$-modules associated with a local field $k$, providing a formal-group-theoretic approach to local class field theory. Building on this theory, Coleman\cite{Coleman} introduced trace and norm operators attached to Lubin--Tate formal groups and used them to study norm-compatible systems arising from division points.

Independently, Witt\cite{Witt} introduced the $p$-typical Witt vectors, giving a functorial construction which encodes arithmetic in
mixed characteristic through the ghost components.  Cartier\cite{Cartier} subsequently developed the theory of generalized, or big, Witt vectors and clarified its relation with formal groups and formal curves.

The starting point of the present work was a direct computation of the Coleman norm of the Artin--Hasse curve associated with a $p$-typical Witt vector.

Let $p$ be an odd prime and let
\[
v=\{v_j\}_{j\in\mathbb N}\in W_p(\mathbb Z_p).
\]
Write
\[
g_j=\sum_{r=0}^{j}p^rv_r^{p^{j-r}}~~~(j\in\mathbb N)
\]
for its ghost components.  If
\[
E_p(T)=\exp\left(\sum_{r=0}^{\infty}\frac{T^{p^r}}{p^r}\right)
\]
denotes the Artin--Hasse exponential, then the Artin--Hasse curve associated with $v$ is
\[
f(T)=\prod_{j=0}^{\infty}E_p(v_jT^{p^j})=\exp\left(\sum_{j=0}^{\infty}\frac{g_j}{p^j}T^{p^j}\right)\in 1+T\mathbb Z_p[[T]].
\]

Let $N_{Co}$ denote the Coleman norm associated with the multiplicative Lubin--Tate formal group $\widehat{\mathbb G}_m$ over $\mathbb Q_p$. A direct calculation gives the following formula.

\begin{lem}\label{lem:intro-Coleman}
For every $i\in\mathbb N_{>0}$,
\begin{align*}
(N_{Co}^if)(T)=\exp\Bigg(&-\sum_{j=0}^{i-1}p^{i-j}g_j \\
&+\sum_{j=i}^{\infty}p^{i-j}g_j\sum_{s=0}^{p^{j-i}}(-1)^{p^{j-i}-s}\binom{p^j}{p^is} \\
&+\sum_{t=1}^{\infty}\Bigg\{\sum_{j=i}^{\infty}p^{i-j}g_j\sum_{s=1}^{p^{j-i}}(-1)^{p^{j-i}-s}\binom{p^j}{p^is}\binom{s}{t}\Bigg\}T^t\Bigg).
\end{align*}
\end{lem}

\begin{proof}
Since $\widehat{\mathbb G}_m[p^i]=\{\zeta_{p^i}^a-1\mid 0\le a<p^i\}$, we have
\begin{align*}
(N_{Co}^if)((1+T)^{p^i}-1)&=\prod_{a=0}^{p^i-1}f\bigl(\zeta_{p^i}^a(1+T)-1\bigr) \\
&=\exp\left(\sum_{j=0}^{\infty}\frac{g_j}{p^j}\sum_{a=0}^{p^i-1}\bigl(\zeta_{p^i}^a(1+T)-1\bigr)^{p^j}\right).
\end{align*}
Expanding the inner power and using
\[
\sum_{a=0}^{p^i-1}\zeta_{p^i}^{a\ell}=\begin{cases}p^i & (p^i\mid\ell),\\ 0 & (p^i\nmid\ell),\end{cases}
\]
we obtain
\[
\sum_{a=0}^{p^i-1}\bigl(\zeta_{p^i}^a(1+T)-1\bigr)^{p^j}=\begin{cases}-p^i & (j<i), \\[6pt] \ds p^i\sum_{s=0}^{p^{j-i}}(-1)^{p^{j-i}-s}\binom{p^j}{p^is}(1+T)^{p^is} & (j\ge i). \end{cases}
\]
Here we have used the fact that $p$ is odd.  Consequently,
\begin{align*}
&(N_{Co}^if)((1+T)^{p^i}-1) \\
=&\exp\Bigg(-\sum_{j=0}^{i-1}p^{i-j}g_j+\sum_{j=i}^{\infty}p^{i-j}g_j\sum_{s=0}^{p^{j-i}}(-1)^{p^{j-i}-s}\binom{p^j}{p^is}(1+T)^{p^is}\Bigg).
\end{align*}
Now
\[
(1+T)^{p^is}=\left(1+\bigl((1+T)^{p^i}-1\bigr)\right)^s=\sum_{t=0}^{s}\binom{s}{t}\bigl((1+T)^{p^i}-1\bigr)^t.
\]
Collecting the coefficients of $\bigl((1+T)^{p^i}-1\bigr)^t$ and using the fact that
\[
[p^i]_{\widehat{\mathbb G}_m}(T)=(1+T)^{p^i}-1
\]
is invertible under composition over $\mathbb Q_p$, we obtain the asserted formula.
\end{proof}

For $i,t\in\mathbb N_{>0}$, denote the coefficient of $T^t$ in $\ds \log\frac{(N_{Co}^if)(T)}{(N_{Co}^if)(0)}$ by $A_{i,t}(v)$. By Lemma~\ref{lem:intro-Coleman}, it is given by
\[
A_{i,t}(v)=\sum_{j=i}^{\infty}p^{i-j}g_j\sum_{s=1}^{p^{j-i}}(-1)^{p^{j-i}-s}\binom{p^j}{p^is}\binom{s}{t}.
\]

The contraction property of the Coleman norm already gives a nontrivial congruence among these coefficients.

\begin{cor}\label{cor:intro-contraction}
For every $i,t\in\mathbb N_{>0}$,
\[
A_{i+1,t}(v)\equiv A_{i,t}(v)\pmod{p^{i+1}\mathbb Z_p}.
\]
In particular, for every fixed $t\in\mathbb N_{>0}$, the sequence
\[
\{A_{i,t}(v)\}_{i\ge1}
\]
is Cauchy in $\mathbb Q_p$ and hence converges $p$-adically.
\end{cor}

\begin{proof}
Put
\[
\widetilde N_{Co}^i(f)=\frac{N_{Co}^i(f)}{(N_{Co}^i(f))(0)}.
\]
By the contraction lemma for the Coleman norm,
\[
\frac{\widetilde N_{Co}^{i+1}(f)}{\widetilde N_{Co}^{i}(f)}\in 1+p^{i+1}T\mathbb Z_p[[T]].
\]
Since $p$ is odd, the $p$-adic logarithm gives
\[
\log\widetilde N_{Co}^{i+1}(f)-\log\widetilde N_{Co}^{i}(f)\in p^{i+1}T\mathbb Z_p[[T]].
\]
By Lemma~\ref{lem:intro-Coleman},
\[
\log\widetilde N_{Co}^{i}(f)=\sum_{t=1}^{\infty}A_{i,t}(v)T^t.
\]
Comparing coefficients of $T^t$ yields
\[
A_{i+1,t}(v)-A_{i,t}(v)\in p^{i+1}\mathbb Z_p.
\]
The convergence follows immediately.
\end{proof}

The finite sums occurring in the coefficients $A_{i,t}(v)$ are polynomial extensions of the classical Fleck sums.  More precisely, for $j\ge i$ put
\[
C_{i,t}(j)=\sum_{s=0}^{p^{j-i}}(-1)^{p^{j-i}-s}\binom{p^j}{p^is}\binom{s}{t}.
\]
Then
\[
A_{i,t}(v)=\sum_{j=i}^{\infty}p^{i-j}g_jC_{i,t}(j)~~~(t\ge1).
\]
Since $p$ is odd,
\[
C_{i,t}(j)=(-1)^{p^j}\sum_{\substack{ 0\le a\le p^j\\p^i\mid a}}(-1)^a\binom{p^j}{a}\binom{a/p^i}{t}.
\]
$C_{i,t}(j)$ is called a polynomially weighted Fleck sum.

Fleck's classical congruence\cite{Fleck} was generalized to prime-power moduli by Weisman\cite{Weisman}.  Polynomial extensions were subsequently studied by Sun\cite{Sun}, and Wan\cite{Wan} obtained sharp estimates by means of the $\psi$-operator in the cyclotomic setting.

Throughout the following theorem, for $x\in\mathbb R$ we write
\[
\lfloor x\rfloor_+=\max\{0,\lfloor x\rfloor\}.
\]

\begin{thm}[Wan\cite{Wan}, Theorems 1.3 and 1.4]\label{thm:intro-Wan}
Let $i\in\mathbb N_{>0}$.

\begin{enumerate}
\item[\textnormal{(i)}]
For $n\in\mathbb N_{>0}$, $t\in\mathbb N$, and $r\in\mathbb Z$,
\[
v_p\left(\sum_{\substack{0\le k\le n\\ k\equiv r\pmod{p^i}}}(-1)^{n-k}\binom nk\binom{(k-r)/p^i}{t}\right)\ge\left\lfloor\frac{n-p^{i-1}-t p^i}{p^{i-1}(p-1)}\right\rfloor_+.
\]

\item[\textnormal{(ii)}]
For $n\in\mathbb N_{>0}$, $t\in\mathbb N$, and $r\in\mathbb Z$,
\begin{align*}
v_p\Bigg(&\sum_{\substack{s_0,\ldots,s_{p^i-1}\ge0\\ s_0+\cdots+s_{p^i-1}=n\\ s_1+2s_2+\cdots+(p^i-1)s_{p^i-1}\equiv r\pmod{p^i}}} \\
&\binom{n}{s_0,\ldots,s_{p^i-1}}\binom{(s_1+2s_2+\cdots+(p^i-1)s_{p^i-1}-r)/p^i}{t}\Bigg) \\
&\qquad\qquad\qquad\qquad\ge\left\lfloor\frac{(in-i+1)(p-1)-t(ip-i+1)-1}{p-1}\right\rfloor_+.
\end{align*}
\end{enumerate}
\end{thm}

Applying part~(i) with
\[
n=p^j,\qquad r=0
\]
we obtain
\[
v_p\bigl(C_{i,t}(j)\bigr)\ge\left\lfloor\frac{p^j-p^{i-1}-tp^i}{p^{i-1}(p-1)}\right\rfloor_+.
\]
Part~(ii) is the corresponding estimate arising from negative powers. Both types of sums will reappear below as specializations to $F=\widehat{\mathbb G}_m$ of the Fleck operators associated with $T^n$ and $T^{-n}$, respectively.

On the other hand, Hazewinkel\cite{Hazewinkel} constructed ramified $q$-typical Witt vectors associated with $\sigma$-twisted Lubin--Tate formal groups. He proved the integrality of the corresponding Witt addition and multiplication polynomials, thereby obtaining the ramified $q$-typical Witt functor.

The preceding results suggest that the relation observed in Lemma~\ref{lem:intro-Coleman} should admit a Lubin--Tate generalization on both sides.  Indeed, the formula involves two different structures: the $p$-typical Witt ghost components and polynomially weighted Fleck sums.  Hazewinkel's theory provides a Lubin--Tate generalization of the former, while the results of Sun and Wan suggest a corresponding generalization of the latter.

The purpose of the present paper is to generalize this relation. We first extend Hazewinkel's ramified $q$-typical Witt theory to a big Witt theory associated with a suitable $\sigma$-twisted Lubin--Tate formal group.  We then introduce Lubin--Tate analogues of the polynomial Fleck sums and prove analogues of the Fleck--Sun--Wan congruences.  Finally, we construct a Coleman $F$-norm on the resulting big Witt ring and show that its action on ghost components is governed by these generalized Fleck sums.

We now state the main results of this paper.

Our first main result concerns the Witt-theoretic side of Lemma~\ref{lem:intro-Coleman}.

\begin{thm}[Construction of the $F$-big Witt ring]\label{thm:intro-F-big-Witt}
Let $F$ be a coefficient-multiplicative normalized $\sigma$-twisted Lubin--Tate formal group. Then there exists a functor
\[
B\longmapsto W^F(B)
\]
from $R$-algebras to commutative $R$-algebras, equipped with natural analogues of the ghost map, Teichm\"uller map, Verschiebung, and Frobenius. It also admits a natural realization in terms of Lubin--Tate curves.

Moreover, for $F=\widehat{\mathbb G}_m$ over $\mathbb Q_p$, this construction recovers the usual big Witt ring, and its $q$-typical part recovers Hazewinkel's ramified $q$-typical Witt ring.
\end{thm}

A precise formulation of Theorem~\ref{thm:intro-F-big-Witt} will be given in Theorem~\ref{thm:F-big-Witt} below.

We next turn to the Fleck-theoretic side of Lemma~\ref{lem:intro-Coleman}. Let $K$ be a local field with residue field of cardinality $q$, and let $F$ be a Lubin--Tate formal $\mathfrak o_K$-module. If $\operatorname{Tr}_{Co}$ denotes the Coleman trace, put
\[
\mathfrak F_{s,m}^F(g)=(\operatorname{the~coefficient~of~}T^m\text{~in~}\operatorname{Tr}_{Co}^s(g)).
\]
We shall later refer to the operator $\mathfrak{F}_{s,m}^F$ as the Fleck operator.

Our second main result is the following Lubin--Tate analogue of the Fleck--Sun--Wan congruences.

\begin{thm}\label{general-FSW}
Let $s,n\in\mathbb N_{>0}$.
\begin{enumerate}
\item[\textnormal{(i)}]
For every $m\in\mathbb N$,
\[
v_K\left(\mathfrak F_{s,m}^F(T^n)\right)\ge\left\lfloor\frac{n-q^{s-1}-mq^s}{(q-1)q^{s-1}}\right\rfloor_+.
\]

\item[\textnormal{(ii)}]
For every $m\in\mathbb Z$ with $m\ge -n$,
\[
v_K\left(\mathfrak F_{s,m}^F(T^{-n})\right)\ge\left\lfloor\frac{(q-1)(1-s-sm)-m-n-1}{q-1}\right\rfloor_+.
\]
\end{enumerate}
\end{thm}

When $K=\mathbb Q_p$ and $F=\widehat{\mathbb G}_m$, the quantities $\mathfrak F_{s,m}^F(T^n)$ and $\mathfrak F_{s,m}^F(T^{-n})$ specialize respectively to the two types of polynomially weighted Fleck sums appearing in the $r=0$ cases of Theorem~\ref{thm:intro-Wan}. Thus the theorem recovers the corresponding cases of Wan's estimates.

Our third main result reconnects the Witt-theoretic and Fleck-theoretic generalizations above. We construct a Coleman $F$-norm on the $F$-big Witt ring and prove a Cartier-type expansion whose coefficients are expressed in terms of the $F$-Fleck operators. On ghost components, this becomes a convergent linear operator with generalized Fleck coefficients as its matrix coefficients. The integrality of the Cartier coefficients yields congruences relating distinct $F$-Fleck coefficients. In particular, when $F=\widehat{\mathbb G}_m$, we obtain the following congruence involving classical polynomially weighted Fleck sums.

\begin{cor}\label{cor:intro-new-congruence}
Let $p$ be an odd prime and let $m,n\in\mathbb N_{>0}$. Then
\begin{align*}
v_p\left(\sum_{d\mid(m,n)}\mu(d)(-1)^{n/d}\sum_{\substack{0\le a\le n/d\\p\mid a}}(-1)^a\binom{n/d}{a}\binom{a/p}{m/d}\right)\ge v_p(n)-1.
\end{align*}
Consequently, if $v_p(n)\ge2$, then
\[
\sum_{d\mid(m,n)}\mu(d)(-1)^{n/d}\sum_{\substack{0\le a\le n/d\\p\mid a}}(-1)^a\binom{n/d}{a}\binom{a/p}{m/d}\equiv 0\pmod{p^{v_p(n)-1}}.
\]
\end{cor}

\section{$\sigma$-twisted Lubin--Tate formal logarithms}

Let $p$ be an odd prime, let $k$ be a finite Galois extension of $\mathbb{Q}_p$ with valuation ring $\ded_k$ and a uniformizer $\pi$, and let $\sigma\in\operatorname{Gal}(k/\mathbb{Q}_p)$ be a lift of the $q=p^{f_\sigma}$-power Frobenius, so that
\[
\sigma\alpha\equiv \alpha^q\pmod{\pi}.
\]
Such lifts $\sigma$ are unique modulo the inertia group. We write $\mathbb{N}_{(p)}\subset \mathbb{N}_{>0}$ for the monoid generated by all primes different from $p$.

Hazewinkel\cite{Hazewinkel} introduced $\sigma$-twisted Lubin--Tate formal groups associated with $\pi$ and established the corresponding functional-equation lemma.

\begin{defn}
Let $R$ be a $\pi$-torsion-free $\ded_k[\langle\sigma\rangle]$-algebra such that $\sigma$ is a lift of the $q$-power Frobenius modulo $\pi$, i.e.,
\[
\sigma(\alpha)\equiv \alpha^q\pmod{\pi R}~~~(\alpha\in R).
\]
We also put
\[
K=k\otimes_{\ded_k}R.
\]

Let $\ell(T)\in TK[[T]]$. We call $\ell$ a \textit{$\sigma$-twisted Lubin--Tate formal logarithm over $R$ associated with $\pi$} if
\[
\ell(T)-\frac{1}{\pi}\sigma(\ell(T^q))\in TR[[T]].
\]
We denote the set of all such logarithms by $\mathcal{LT}_R$. If, in addition,
\[
\ell(T)\equiv T\pmod{T^2},
\]
then $\ell$ is called a \textit{normalized $\sigma$-twisted Lubin--Tate formal logarithm over $R$ associated with $\pi$}.
\end{defn}

\begin{lem}\label{FEL}
Let
\[
\ell(T)=\sum_{n=1}^{\infty}b_nT^n ~~~(b_1=1,\ b_n\in K)
\]
be a normalized $\sigma$-twisted Lubin--Tate formal logarithm over $R$ associated with $\pi$. Then the following assertions hold.
\begin{enumerate}
\item[\textnormal{(i)}]
For every $n=q^im\in\mathbb{N}_{>0},~~q\nmid m$, we have
\[
b_n\in\pi^{-i}R.
\]

\item[\textnormal{(ii)}]
For every $h\in R[[X_1,\dots,X_r]]$ and every $n=q^im\in\mathbb{N}_{>0}~~~(q\nmid m,~~i\ge 1)$, we have
\[
h(X_1,\dots,X_r)^n\equiv (\sigma h)(X_1^q,\dots,X_r^q)^{n/q}\pmod{\pi^iR[[X_1,\dots,X_r]]}.
\]

\item[\textnormal{(iii)}]
For every $g\in (X_1,\dots,X_r)R[[X_1,\dots,X_r]]$, we have
\[
(\ell\circ g)(X_1,\dots,X_r)-\frac{1}{\pi}(\sigma(\ell\circ g))(X_1^q,\dots,X_r^q)\in R[[X_1,\dots,X_r]].
\]

\item[\textnormal{(iv)}]
Suppose that $h\in (X_1,\dots,X_r)K[[X_1,\dots,X_r]]$ satisfies
\[
h(X_1,\dots,X_r)-\frac{1}{\pi}(\sigma h)(X_1^q,\dots,X_r^q)\in R[[X_1,\dots,X_r]].
\]
Then
\[
\ell^{-1}\circ h\in (X_1,\dots,X_r)R[[X_1,\dots,X_r]].
\]
\end{enumerate}
\end{lem}

\begin{proof}
(i)Since $\ds \ell(T)-\frac{1}{\pi}\sigma(\ell(T^q))\in TR[[T]]$, we have $q\nmid n\Longrightarrow b_n\in R$, and $\ds q\mid n\Longrightarrow b_n-\frac{1}{\pi}\sigma(b_{n/q})\in R$. The assertion is clear for $i=0$. Suppose that it holds for $i-1$. Then
\[
b_n\in R+\frac{1}{\pi}\sigma(b_{n/q})\subset R+\pi^{-1}\sigma(\pi^{-(i-1)}R)\subset R+\pi^{-i}R=\pi^{-i}R.
\]
Hence, by induction,
\[
b_n\in\pi^{-i}R~~~(\forall n=q^im\in\mathbb{N}_{>0}).
\]

\medskip

(ii)Write
\[
h(X_1,\dots,X_r)^m=\sum_{\mathbf{j}=\mathbf{0}}^\infty a_{\mathbf{j}}X_1^{j_1}\dots X_r^{j_r}.
\]
Then
\begin{align*}
h(X_1,\dots,X_r)^{qm}&\equiv \sum_{\mathbf{j}=\mathbf{0}}^\infty a_{\mathbf{j}}^qX_1^{qj_1}\dots X_r^{qj_r}\equiv\sum_{\mathbf{j}=\mathbf{0}}^\infty \sigma(a_{\mathbf{j}})X_1^{qj_1}\dots X_r^{qj_r} \pmod{\pi R} \\
&=(\sigma h)(X_1^q,\dots,X_r^q)^m.
\end{align*}
Thus the assertion holds for $i=1$. Assume that it holds for $i-1$. Then
\[
h(X_1,\dots,X_r)^{n/q}-(\sigma h)(X_1^q,\dots,X_r^q)^{n/q^2}\in \pi^{i-1}R[[X_1,\dots,X_r]].
\]
Write $\pi^{i-1}H(X_1,\dots,X_r)=h(X_1,\dots,X_r)^{n/q}-(\sigma h)(X_1^q,\dots,X_r^q)^{n/q^2}$, then
\begin{align*}
&h(X_1,\dots,X_r)^n-(\sigma h)(X_1^q,\dots,X_r^q)^{n/q} \\
=&\left\{\pi^{i-1}H(X_1,\dots,X_r)+(\sigma h)(X_1^q,\dots,X_r^q)^{n/q^2}\right\}^q-(\sigma h)(X_1^q,\dots,X_r^q)^{n/q} \\
=&\sum_{s=1}^{q}\binom{q}{s} \pi^{(i-1)s}H(X_1,\dots,X_r)^s(\sigma h)(X_1^q,\dots,X_r^q)^{\frac{n}{q^2}(q-s)} \\
\in &\pi^iR[[X_1,\dots,X_r]].
\end{align*}
Hence the assertion follows by induction.

\medskip

(iii)By (i) and (ii),
\begin{align*}
(\ell\circ g)(X_1,\dots,X_r)&\equiv \sum_{q\mid n}b_ng(X_1,\dots,X_r)^n\equiv\frac{1}{\pi}\sum_{q\mid n}\sigma(b_{n/q})g(X_1,\dots,X_r)^n \\
&\equiv\frac{1}{\pi}\sum_{q\mid n}\sigma(b_{n/q})\sigma(g)(X_1^q,\dots,X_r^q)^{n/q} \\
&=\frac{1}{\pi}(\sigma(\ell\circ g))(X_1^q,\dots,X_r^q)\pmod{R[[X_1,\dots,X_r]]}.
\end{align*}
Therefore
\[
(\ell\circ g)(X_1,\dots,X_r)-\frac{1}{\pi}(\sigma(\ell\circ g))(X_1^q,\dots,X_r^q)\in R[[X_1,\dots,X_r]].
\]

\medskip

(iv)Write $\ds (\ell^{-1}\circ h)(X_1,\dots,X_r)=\sum_{d=1}^{\infty}G_d(X_1,\dots,X_r)$, where $G_d(X_1,\dots,X_r)$ is homogeneous of degree $d$. Since $\ell^{-1}(T)\equiv T\pmod{T^2}$, we have
\begin{align*}
G_1(X_1,\dots,X_r)&\equiv (\ell^{-1}\circ h)(X_1,\dots,X_r)\equiv h(X_1,\dots,X_r) \\
&\equiv\frac{1}{\pi}(\sigma h)(X_1^q,\dots,X_r^q)\equiv 0
\end{align*}
modulo $R[[X_1,\dots,X_r]]+(X_1,\dots,X_r)^2K[[X_1,\dots,X_r]]$. Hence $G_1\in R[X_1,\dots,X_r]$. Suppose that the assertion holds in degrees less than $d$, and put
\[
g_d=\sum_{d'=1}^{d-1}G_{d'}\in R[[X_1,\dots,X_r]].
\]
By (iii),
\[
(\ell\circ g_d)(X_1,\dots,X_r)-\frac{1}{\pi}(\sigma(\ell\circ g_d))(X_1^q,\dots,X_r^q)\in R[[X_1,\dots,X_r]].
\]
Put $h_d=h-\ell\circ g_d$. Then
\[
h_d(X_1,\dots,X_r)-\frac{1}{\pi}(\sigma h_d)(X_1^q,\dots,X_r^q)\in R[[X_1,\dots,X_r]].
\]
Since $h_d$ has no terms of degree less than $d$,
the series $h_d(X_1^q,\dots,X_r^q)$ has degree at least $qd>d$.
Moreover,
\begin{align*}
G_d(X_1,\dots,X_r)&\equiv (\ell^{-1}\circ h)(X_1,\dots,X_r)-g_d(X_1,\dots,X_r) \\
&\equiv (\ell\circ\ell^{-1}\circ h)(X_1,\dots,X_r)-(\ell\circ g_d)(X_1,\dots,X_r) \\
&=h(X_1,\dots,X_r)-(\ell\circ g_d)(X_1,\dots,X_r) \\
&=h_d(X_1,\dots,X_r) \\
&\equiv\frac{1}{\pi}(\sigma h_d)(X_1^q,\dots,X_r^q)\equiv 0
\end{align*}
modulo $R[[X_1,\dots,X_r]]+(X_1,\dots,X_r)^{d+1}K[[X_1,\dots,X_r]]$. Thus
\[
G_d(X_1,\dots,X_r)\in R[X_1,\dots,X_r].
\]
The assertion follows by induction.
\end{proof}

\begin{cor}\label{FEL-cor}
Let $\ell$ be a normalized $\sigma$-twisted Lubin--Tate formal
logarithm over $R$ associated with $\pi$, and define
\[
\widetilde{\ell}:TK[[T]]\xrightarrow{\sim}TK[[T]];~g\longmapsto\ell\circ g.
\]
Then the following hold.

\begin{enumerate}
\item[\textnormal{(i)}]
The power series
\[
F(X,Y)=\ell^{-1}(\ell(X)+\ell(Y))
\]
defines a formal group over $R$.

\item[\textnormal{(ii)}]
The map
\[
\widetilde{\ell}:TR[[T]]\xrightarrow{\sim}\mathcal{LT}_R;~g\longmapsto\ell\circ g,
\]
is a bijection.

\item[\textnormal{(iii)}]
We have
\[
\widetilde{\ell}^{-1}(\mathcal{LT}_R)=TR[[T]].
\]

\item[\textnormal{(iv)}]
For every $\alpha\in TR[[T]]$, $\beta\in TK[[T]]$, and $s,t\in\mathbb{N}_{>0}$,
\[
\alpha\equiv\beta\pmod{\pi^sR[[T]]+T^tK[[T]]}
\]
if and only if
\[
\ell\circ\alpha\equiv\ell\circ\beta \pmod{\pi^sR[[T]]+T^tK[[T]]}.
\]
\end{enumerate}
\end{cor}

\begin{proof}
(i)Since $\ell(T)\equiv T\pmod{T^2}$, there exists $\ell^{-1}(T)\in T+T^2K[[T]]$. Hence $F(X,Y)=\ell^{-1}(\ell(X)+\ell(Y))\in K[[X,Y]]$ and
\[
F(X,Y)\equiv X+Y\pmod{(X,Y)^2K[[X,Y]]}.
\]
Moreover,
\begin{align*}
F(F(X,Y),Z)&=\ell^{-1}\bigl(\ell(X)+\ell(Y)+\ell(Z)\bigr) \\
&=F(X,F(Y,Z)),
\end{align*}
and
\[
F(X,0)=X,\qquad F(0,X)=X,
\]
while
\[
F(X,Y)=F(Y,X).
\]
Thus $F$ is a commutative formal group over $K$. Furthermore,
\[
(\ell(X)+\ell(Y))-\frac{1}{\pi}\sigma(\ell(X^q)+\ell(Y^q))\in R[[X,Y]].
\]
Therefore, by part (iv) of the functional-equation lemma, $F(X,Y)\in R[[X,Y]]$. Hence $F$ is defined over $R$.

\medskip

(ii) and (iii) follow immediately from parts (iii) and (iv) of the functional-equation lemma.

\medskip

(iv)The case $t=1$ is immediate. Assume $t>1$. Suppose first that
\[
\alpha\equiv\beta\pmod{\pi^sR[[T]]+T^tK[[T]]}.
\]
Let $\widetilde{\alpha}_t$ and $\widetilde{\beta}_t$ denote the truncations of $\alpha$ and $\beta$, respectively, to terms of degree at most $t-1$. Then
\[
\widetilde{\beta}_t\equiv\widetilde{\alpha}_t\pmod{\pi^sR[[T]]},
\]
and hence $\widetilde{\beta}_t\in R[[T]]$. By the binomial theorem,
\[
\widetilde{\alpha}_t^{q^i}\equiv\widetilde{\beta}_t^{q^i}\pmod{\pi^{s+i}R[[T]]},
\]
and consequently, for $n=q^im$ with $q\nmid m$,
\[
\widetilde{\alpha}_t^n\equiv\widetilde{\beta}_t^n\pmod{\pi^{s+i}R[[T]]}.
\]
By part (i) of the functional-equation lemma,
\begin{align*}
\ell\circ\alpha&\equiv\ell\circ\widetilde{\alpha}_t=\sum_{n=1}^{\infty}b_n\widetilde{\alpha}_t^n\equiv\sum_{n=1}^{\infty}b_n\widetilde{\beta}_t^n=\ell\circ\widetilde{\beta}_t \\
&\equiv\ell\circ\beta\pmod{\pi^sR[[T]]+T^tK[[T]]}.
\end{align*}

Conversely, suppose that
\[
\ell\circ\alpha\equiv\ell\circ\beta\pmod{\pi^sR[[T]]+T^tK[[T]]}.
\]
Write
\[
\alpha(T)=\sum_{j=1}^{\infty}\alpha_jT^j,~~\beta(T)=\sum_{j=1}^{\infty}\beta_jT^j.
\]
Since $\ell(T)\equiv T\pmod{T^2}$, we obtain $\alpha_1\equiv\beta_1\pmod{\pi^sR}$. Suppose inductively that, for some $d<t$,
\[
\alpha_j\equiv\beta_j\pmod{\pi^sR}~~~(j=1,\dots,d-1).
\]
Let $\widetilde{\alpha}_d$ and $\widetilde{\beta}_d$ denote the truncations to degree at most $d-1$. Then
\[
\widetilde{\alpha}_d\equiv\widetilde{\beta}_d\pmod{\pi^sR[[T]]}.
\]
For every $j=q^im\in\mathbb{N}_{>0}$ with $q\nmid m$,
\[
\widetilde{\alpha}_d^j\equiv\widetilde{\beta}_d^j\pmod{\pi^{s+i}R[[T]]}.
\]
Let $A_{d,j}$ and $B_{d,j}$ denote the coefficients of $T^d$ in $\widetilde{\alpha}_d^j$ and $\widetilde{\beta}_d^j$, respectively.
Comparing the coefficients of $T^d$ in $\ell\circ\alpha-\ell\circ\beta$, we obtain
\[
0\equiv(\alpha_d-\beta_d)+\sum_{j=1}^{d}b_j(A_{d,j}-B_{d,j})\equiv\alpha_d-\beta_d\pmod{\pi^sR}.
\]
Thus $\alpha_d\equiv\beta_d\pmod{\pi^sR}$. By induction,
\[
\alpha_d\equiv\beta_d\pmod{\pi^sR}~~~(d=1,\dots,t-1),
\]
and hence
\[
\alpha\equiv\beta\pmod{\pi^sR[[T]]+T^tK[[T]]}.
\]
\end{proof}

\begin{defn}
Let $\ell$ be a normalized $\sigma$-twisted Lubin--Tate formal logarithm over $R$ associated with $\pi$. The formal group
\[
F(X,Y)=\ell^{-1}(\ell(X)+\ell(Y))
\]
is called the \textit{$\sigma$-twisted Lubin--Tate formal group over $R$ associated with $\pi$ having logarithm $\ell$}.
Then
\[
\ell=\log_F.
\]

Equivalently, a $\sigma$-twisted Lubin--Tate formal group over $R$ associated with $\pi$ is a formal group $F$ over $R$ such that
\[
\log_F(T)-\frac{1}{\pi}\sigma(\log_F(T^q))\in TR[[T]].
\]
\end{defn}

\begin{lem}
Let $F$ be a $\sigma$-twisted Lubin--Tate formal group over $R$ associated with $\pi$, and define
\[
[\cdot]_F:K\longrightarrow\operatorname{End}_K(F);~\alpha\longmapsto \exp_F(\alpha\log_F(T)).
\]
Then the following hold.

\begin{enumerate}
\item[\textnormal{(i)}]
If $\alpha\in R^{\langle\sigma\rangle}$, then
\[
[\alpha]_F(T)\in\operatorname{End}_R(F).
\]

\item[\textnormal{(ii)}]
Suppose that $K/k$ is a finite unramified extension, $R=\ded_K$, $q=\#\kappa_K$, and $\sigma=\id$. Then $(F,[\cdot]_F)$ is a Lubin--Tate formal $\ded_K$-module.
\end{enumerate}
\end{lem}

\begin{proof}
(i)For every $\alpha\in R^{\langle\sigma\rangle}$,
\[
\alpha\log_F(T)-\frac{1}{\pi}\sigma(\alpha\log_F(T^q))=\alpha\left(\log_F(T)-\frac{1}{\pi}(\sigma\log_F)(T^q)\right)\in R[[T]].
\]
Hence, by part (iv) of the functional-equation lemma,
\[
[\alpha]_F(T)=\exp_F(\alpha\log_F(T))\in R[[T]].
\]
Thus $[\alpha]_F\in\operatorname{End}_R(F)$.

\medskip

(ii)Since $R^{\langle\sigma\rangle}=R$, we obtain
\[
[\cdot]_F:\ded_K\to \operatorname{End}_{\ded_K}(F),
\]
so $(F,[\cdot]_F)$ is a formal $\ded_K$-module.

Moreover, since $\log_F(T)\equiv T\pmod{T^2}$, we have
\[
[\pi]_F(T)=\exp_F(\pi\log_F(T))\equiv \pi T\pmod{T^2}.
\]
Also,
\begin{align*}
\log_F([\pi]_F(T))-\log_F(T^q)&=\pi\log_F(T)-\log_F(T^q) \\
&=\pi\left(\log_F(T)-\frac{1}{\pi}\log_F(T^q)\right)\in \pi\ded_K[[T]].
\end{align*}
By part (iv) of the preceding corollary, for every $n\in\mathbb{N}_{>0}$, $[\pi]_F(T)\equiv T^q
\pmod{\pi\ded_K[[T]]+T^nK[[T]]}$. It follows that
\[
[\pi]_F(T)\equiv T^q\pmod{\pi\ded_K[[T]]}.
\]
Therefore $(F,[\cdot]_F)$ is a Lubin--Tate formal $\ded_K$-module.
\end{proof}

\begin{defn}
Let $M\subset\mathbb{N}_{>0}$ be a submonoid. We call $M$ \textit{monus-closed} if
\[
m,n\in M~\Longrightarrow~\frac{m}{(m,n)}\in M.
\]

Let $N_{(p)}\subset\mathbb{N}_{(p)}$ be a monus-closed submonoid, let
\[
\rho:N_{(p)}\longrightarrow (R^\ast)^{\langle\sigma\rangle}
\]
be a multiplicative map, and let $f_Ff_{\sigma/F}=f_\sigma~~(f_F,f_{\sigma/F}\in\mathbb{N}_{>0})$. Choose $1=\nu_0,\nu_1,\dots,\nu_{f_{\sigma/F}-1}\in (R^\ast)^{\langle\sigma\rangle}$, and put
\[
N=\langle p^{f_F}\rangle\times N_{(p)}.
\]
For $i\in\mathbb N$, define
\[
\Pi_i=\prod_{j=0}^{i-1}\sigma^j(\pi),
\]
and, for
\[
0\le j<f_{\sigma/F},~~ i\in\mathbb N,~~ m\in N_{(p)},
\]
define
\[
c_{p^{f_Fj}q^im}=\Pi_i\nu_j\rho(m).
\]
Finally, put
\[
\ell(T)=\sum_{n\in N}\frac{(-1)^{n-1}}{c_n}T^n\in TK[[T]].
\]
We call $\ell$ a \textit{coefficient-multiplicative normalized $\sigma$-twisted Lubin--Tate formal logarithm over $R$ associated with $\pi$}.

We also write
\[
v_q:\mathbb{N}_{>0}\longrightarrow\mathbb{N};~n\longmapsto \left\lfloor\frac{v_p(n)}{f_\sigma}\right\rfloor,
\]
\[
N_{(q)}=\{n\in N~|~q\nmid n\},
\]
and
\[
d\mid_N n~:\Longleftrightarrow~d,n\in N\ \text{and}\ d\mid n.
\]
In this case $\ds \frac{n}{d}\in N$ by the definition of the monus-closedness.

By construction,
\[
c_{qn}=\pi\sigma(c_n)~~~(n\in N).
\]
Hence
\begin{align*}
\ell(T)-\frac{1}{\pi}\sigma(\ell(T^q))&=\sum_{n\in N}\frac{(-1)^{n-1}}{c_n}T^n-\sum_{n\in N}\frac{(-1)^{n-1}}{\pi\sigma(c_n)}T^{qn} \\
&=\sum_{n\in N}\frac{(-1)^{n-1}}{c_n}T^n-\sum_{n\in N}\frac{(-1)^{qn-1}}{c_{qn}}T^{qn} \\
&=\sum_{n\in N_{(q)}}\frac{(-1)^{n-1}}{c_n}T^n\in R[[T]].
\end{align*}
Thus $\ell$ is a normalized $\sigma$-twisted Lubin--Tate formal logarithm over $R$ associated with $\pi$.

The corresponding normalized $\sigma$-twisted Lubin--Tate formal group $F$ over $R$ associated with $\pi$ is called a \textit{coefficient-multiplicative normalized $\sigma$-twisted Lubin--Tate formal group over $R$ associated with $\pi$}.

Moreover,
\[
\ell_q(T)=\sum_{i=0}^{\infty}\frac{1}{c_{q^i}}T^{q^i}\in TK[[T]]
\]
is a normalized $\sigma$-twisted Lubin--Tate formal logarithm associated with $\pi$. We denote the corresponding formal group by $F_q$.
\end{defn}

\section{Formal theory of $F$-big Witt vectors}

In this section, we develop the formal theory underlying the construction of the $F$-big Witt ring.

\subsection{Basic Transformations}

Let
\[
\mathbb{V}=\{V_n\}_{n\in N},~~\mathbb{X}=\{X_n\}_{n\in N},~~\mathbb{C}=\{C_n\}_{n\in\mathbb{N}_{>0}},
\]
\[
\mathbb{V}_q=\{V_{q,i}\}_{i\in\mathbb{N}}
\]
be families of indeterminates.

We regard $\mathbb{V}$ as a Witt vector, $\mathbb{X}$ as a family of ghost components, $\mathbb{C}$ as the family of coefficients of a Lubin--Tate curve, and $\mathbb{V}_q$ as a $q$-typical Witt vector.

\begin{defn}
Let $\widetilde{\mathbb{V}}=\{\widetilde{V}_n\}_{n\in\mathbb{N}_{>0}}$ be a family of indeterminates. We write
\[
E^F(\widetilde{\mathbb{V}};T)=\sum_{n=1}^{\infty}\widetilde{E}_n^F(\widetilde{V}_1,\dots,\widetilde{V}_n)T^n=\sum_{n=1}^{\infty}\hspace{-1mm}{{}^F} (-_F)(-\widetilde{V}_nT^n)\in R[\widetilde{\mathbb{V}}][[T]].
\]
Here and throughout, $\ds \sum\hspace{-1mm}{{}^F}$ denotes summation with respect to the formal group law $F$. We call $\widetilde{E}_n^F$ the \textit{$n$-th formal $F$-Euler product expansion polynomial}.

It follows inductively that $\widetilde{E}_n^F$ is of the form
\[
\widetilde{E}_n^F(\widetilde{V}_1,\dots,\widetilde{V}_n)=\widetilde{V}_n+\widetilde{E}_n^{\prime F}(\widetilde{V}_1,\dots,\widetilde{V}_{n-1})
\]
for some polynomial
\[
\widetilde{E}_n^{\prime F}\in R[\widetilde{V}_1,\dots,\widetilde{V}_{n-1}].
\]
We therefore define
\begin{align*}
G_1^F(C_1)&=C_1\in R[C_1], \\
G_n^F(C_1,\dots,C_n)&=C_n-\widetilde{E}_n^{\prime F}\bigl(G_i^F(C_1,\dots,C_i)\mid i=1,\dots,n-1\bigr) \\
&\in R[C_1,\dots,C_n].~~(n\ge 2)
\end{align*}
We call $G_n^F$ the \textit{$n$-th formal $F$-Euler product representation polynomial}.

By definition,
\[
\sum_{n=1}^{\infty}\hspace{-1mm}{{}^F} (-_F)(-\widetilde{V}_nT^n)=\sum_{n=1}^{\infty}\widetilde{E}_n^F(\widetilde{\mathbb{V}})T^n,
\]
and
\[
\sum_{n=1}^{\infty}\hspace{-1mm}{{}^F}(-_F)(-G_n^F(\mathbb{C})T^n)=\sum_{n=1}^{\infty}C_nT^n.
\]

In order to treat $N$-big Witt vectors, we put
\[
E_n^F(\mathbb{V})=\widetilde{E}_n^F(\mathbb{V},\{0\}_{m\notin N})~~~(n\in\mathbb{N}_{>0}).
\]
\end{defn}

\begin{defn}
Put
\[
TK[[T]]_N=\left\{\sum_{n\in N}a_nT^n~\middle|~a_n\in K\right\}.
\]
We define the \textit{formal $F$-coefficient adjustment operator} by
\[
\Theta^F:TK[[T]]_N\xrightarrow{\sim}TK[[T]]_N;~\sum_{n\in N}a_nT^n\longmapsto \sum_{n\in N}c_na_nT^n.
\]
\end{defn}

\begin{defn}
For $n\in N$, define
\[
W_n^F(V_d\mid d\mid_Nn)=\sum_{d\mid_Nn}\frac{c_n}{c_{n/d}}V_d^{n/d}\in R[V_d\mid d\mid_Nn].
\]
We call $W_n^F$ the \textit{$n$-th $F$-Witt ghost polynomial}. We also define
\begin{align*}
&U_1^F(X_1)=X_1 \\
&U_n^F(X_d\mid d\mid_Nn)=\frac{1}{c_n}X_n-\sum_{\substack{d\mid_Nn \\ d<n}}\frac{1}{c_{n/d}}U_d^F(X_{d'}\mid d'\mid_Nd)^{n/d}\in K[X_d\mid d\mid_Nn].
\end{align*}
We call $U_n^F$ the \textit{$n$-th inverse $F$-Witt ghost polynomial}.

For $i\in\mathbb{N}$, we have
\[
W_{q^i}^{F_q}(V_{q,j}~|~j\le i)=\sum_{j=0}^{i}\left(\prod_{s=1}^{j}\sigma^{i-s}(\pi)\right)V_{q,j}^{q^{i-j}},
\]
which is precisely Hazewinkel's $q$-typical ramified Witt polynomial.
\end{defn}

\begin{lem}
We have
\[
\Theta^F\log_F\left(\sum_{n\in N}\hspace{-1mm}{{}^F}(-_F)(-V_nT^n)\right)=\sum_{n\in N}W_n^F(\mathbb{V})T^n.
\]
\end{lem}

\begin{proof}
We have
\begin{align*}
&\Theta^F\log_F\left(\sum_{n\in N}\hspace{-1mm}{{}^F}(-_F)(-V_nT^n)\right)=\Theta^F\left(\sum_{n\in N}-\log_F(-V_nT^n)\right) \\
=&\Theta^F\left(\sum_{n\in N}\sum_{m\in N}\frac{V_n^m}{c_m}T^{nm}\right)=\sum_{n,m\in N}\frac{c_{mn}V_n^m}{c_m}T^{nm} \\
=&\sum_{k\in N}\left(\sum_{n\mid_Nk}\frac{c_k}{c_{k/n}}V_n^{k/n}\right)T^k=\sum_{k\in N}W_k^F(\mathbb{V})T^k.
\end{align*}
\end{proof}

\begin{defn}
For $n\in N$, define
\begin{align*}
A_n^F(V_{q,i}\mid i\le v_q(n))&=U_n^F\left(\{W_{q^i}^{F_q}(\{V_{q,j}\}_{j\le i})\}_{i\le v_q(n)},\{0\}_{d\mid_Nn,\ d\notin\langle q\rangle}\right) \\
&\in K[V_{q,i}\mid i\le v_q(n)].
\end{align*}
We call $A_n^F$ the \textit{$n$-th formal $F$-Artin--Hasse polynomial}.

By definition,
\[
W_n^F(A_d^F(\mathbb{V}_q)\mid d\mid_Nn)=\begin{cases}W_{q^i}^{F_q}(\mathbb{V}_q), & (n=q^i),
\\[3pt]0, & (n\notin\langle q\rangle).\end{cases}
\]
\end{defn}

\subsection{Basic Operations}

\begin{defn}
For $n\in N$, define
\begin{align*}
S_n^F(V_d,V_d'\mid d\mid_Nn)&=U_n^F\left(W_d^F(V_{d'}\mid d'\mid_Nd)+W_d^F(V_{d'}'\mid d'\mid_Nd)\mid d\mid_Nn\right), \\
I_n^F(V_d\mid d\mid_Nn)&=U_n^F\left(-W_d^F(V_{d'}\mid d'\mid_Nd)\mid d\mid_Nn\right), \\
P_n^F(V_d,V_d'\mid d\mid_Nn)&=U_n^F\left(W_d^F(V_{d'}\mid d'\mid_Nd)W_d^F(V_{d'}'\mid d'\mid_Nd)\mid d\mid_Nn\right).
\end{align*}
These belong respectively to
\[
K[V_d,V_d'\mid d\mid_Nn],~~K[V_d\mid d\mid_Nn],~~K[V_d,V_d'\mid d\mid_Nn].
\]

We call $S_n^F$, $I_n^F$, and $P_n^F$ the \textit{$n$-th $F$-Witt addition polynomial}, the \textit{$n$-th $F$-Witt additive inverse polynomial}, and the \textit{$n$-th $F$-Witt multiplication polynomial}, respectively.

Put
\[
\mathbf{0}=\{0\}_{n\in N},~~\mathbf{1}=(1,\{0\}_{1\neq n\in N}).
\]
Then
\begin{align*}
W^F(S^F(\mathbb{V},\mathbb{V}'))&=W^F(\mathbb{V})+W^F(\mathbb{V}'), \\
W^F(I^F(\mathbb{V}))&=-W^F(\mathbb{V}), \\
W^F(P^F(\mathbb{V},\mathbb{V}'))&=W^F(\mathbb{V})W^F(\mathbb{V}'), \\
S^F(S^F(\mathbb{V},\mathbb{V}'),\mathbb{V}'')&=S^F(\mathbb{V},S^F(\mathbb{V}',\mathbb{V}'')), \\
S^F(\mathbb{V},\mathbb{V}')&=S^F(\mathbb{V}',\mathbb{V}), \\
S^F(\mathbb{V},\mathbf{0})&=S^F(\mathbf{0},\mathbb{V})=\mathbb{V}, \\
S^F(\mathbb{V},I^F(\mathbb{V}))&=S^F(I^F(\mathbb{V}),\mathbb{V})=\mathbf{0}, \\
P^F(P^F(\mathbb{V},\mathbb{V}'),\mathbb{V}'')&=P^F(\mathbb{V},P^F(\mathbb{V}',\mathbb{V}'')), \\
P^F(\mathbb{V},\mathbb{V}')&=P^F(\mathbb{V}',\mathbb{V}), \\
P^F(\mathbb{V},\mathbf{1})&=P^F(\mathbf{1},\mathbb{V})=\mathbb{V}, \\
P^F(\mathbb{V},S^F(\mathbb{V}',\mathbb{V}''))&=S^F\bigl(P^F(\mathbb{V},\mathbb{V}'),P^F(\mathbb{V},\mathbb{V}'')\bigr).
\end{align*}

Moreover,
\[
S^F\left((\{V_n\}_{n\in M},\{0\}_{n\notin M}),(\{0\}_{n\in M},\{V_n\}_{n\notin M})\right)=\mathbb{V}
\]
for every subset $M\subset N$, and
\[
P^F((V,0,0,\dots),\mathbb{V}')=\{V^nV_n'\}.
\]

We also have
\[
A^F\bigl(S^{F_q}(\mathbb{V}_q,\mathbb{V}_q')\bigr)=S^F\bigl(A^F(\mathbb{V}_q),A^F(\mathbb{V}_q')\bigr),
\]
and
\[
A^F\bigl(P^{F_q}(\mathbb{V}_q,\mathbb{V}_q')\bigr)=P^F\bigl(A^F(\mathbb{V}_q),A^F(\mathbb{V}_q')\bigr).
\]

All these identities follow formally from the definitions. 
\end{defn}

\begin{defn}
Let $A$ be a commutative ring and let $B$ be an $A[\langle\sigma\rangle]$-algebra. We denote by $B[[\tau;\sigma]]$ the skew power series ring in the variable $\tau$ with respect to $\sigma$, namely
\[
B[[\tau;\sigma]]=\left\{\sum_{i=0}^{\infty}a_i\tau^i~\middle|~a_i\in B\right\},
\]
with addition
\[
\left(\sum_{i=0}^{\infty}a_i\tau^i\right)+\left(\sum_{i=0}^{\infty}b_i\tau^i\right)=\sum_{i=0}^{\infty}(a_i+b_i)\tau^i
\]
and multiplication
\[
\left(\sum_{i=0}^{\infty}a_i\tau^i\right)\left(\sum_{i=0}^{\infty}b_i\tau^i\right)=\sum_{i=0}^{\infty}\left(\sum_{j+k=i}a_j\sigma^j(b_k)\right)\tau^i.
\]
Thus $B[[\tau;\sigma]]$ is in general a noncommutative ring.

Let $\ds u=\sum_{j=0}^{\infty}\alpha_j\tau^j\in K[[\tau;\sigma]]$. For $n\in N$, define
\[
\operatorname{ghost}\text{-}[u]_n^F(\mathbb{X})=\sum_{j=0}^{v_q(n)}\frac{c_n}{c_{n/q^j}}\sigma^{v_q(n)-j}(\alpha_j)X_{n/q^j}\in K[X_d\mid d\mid_Nn],
\]
and
\[
[u]_n^F(\mathbb{V})=U_n^F\bigl(\operatorname{ghost}\text{-}[u]^F(W^F(\mathbb{V}))\bigr)\in K[V_d\mid d\mid_Nn].
\]
We call these respectively the \textit{$n$-th ghost $F$-skew action polynomial associated with $u$} and the \textit{$n$-th $F$-skew action polynomial associated with $u$}.

For $u,v\in K[[\tau;\sigma]]$ and $\alpha\in K$, we have
\begin{align*}
\operatorname{ghost}\text{-}[u]^F(\mathbb{X}+\mathbb{X}')&=\operatorname{ghost}\text{-}[u]^F(\mathbb{X})+\operatorname{ghost}\text{-}[u]^F(\mathbb{X}'), \\
\operatorname{ghost}\text{-}[1]^F(\mathbb{X})&=\mathbb{X}, \\
\operatorname{ghost}\text{-}[u+v]^F(\mathbb{X})&=\operatorname{ghost}\text{-}[u]^F(\mathbb{X})+\operatorname{ghost}\text{-}[v]^F(\mathbb{X}), \\
\operatorname{ghost}\text{-}[uv]^F(\mathbb{X})&=\operatorname{ghost}\text{-}[v]^F\bigl(\operatorname{ghost}\text{-}[u]^F(\mathbb{X})\bigr), \\
\operatorname{ghost}\text{-}[\alpha]^F(\mathbb{X}\mathbb{X}')&=\operatorname{ghost}\text{-}[\alpha]^F(\mathbb{X})\mathbb{X}'=\mathbb{X}\operatorname{ghost}\text{-}[\alpha]^F(\mathbb{X}'), \\
\operatorname{ghost}\text{-}[\alpha]_n^F(\mathbb{X})&=\sigma^{v_q(n)}(\alpha)X_n.
\end{align*}

Moreover,
\begin{align*}
W^F([u]^F(\mathbb{V}))&=\operatorname{ghost}\text{-}[u]^F(W^F(\mathbb{V})), \\
[u]^F(S^F(\mathbb{V},\mathbb{V}'))&=S^F([u]^F(\mathbb{V}),[u]^F(\mathbb{V}')), \\
[1]^F(\mathbb{V})&=\mathbb{V}, \\
[u+v]^F(\mathbb{V})&=S^F([u]^F(\mathbb{V}),[v]^F(\mathbb{V})), \\
[uv]^F(\mathbb{V})&=[v]^F([u]^F(\mathbb{V})), \\
[\alpha]^F(P^F(\mathbb{V},\mathbb{V}'))&=P^F([\alpha]^F(\mathbb{V}),\mathbb{V}')=P^F(\mathbb{V},[\alpha]^F(\mathbb{V}')).
\end{align*}
\end{defn}

\begin{defn}
Let $b\in N$. For $n\in N$, define
\[
V_{b,n}^F(\mathbb{V})=\begin{cases}0 & (b\nmid_Nn), \\ V_{n/b} & (b\mid_Nn),\end{cases}
\]
and
\[
F_{b,n}^F(\mathbb{V})=U_n^F\left(W_{bd}^F(V_{d'}\mid d'\mid_Nbd)\mid d\mid_Nn\right)\in K[V_d\mid d\mid_Nbn].
\]
We call $V_{b,n}^F$ the \textit{$n$-th $F$-$b$-Verschiebung polynomial} and $F_{b,n}^F$ the \textit{$n$-th $F$-$b$-Frobenius polynomial}.

For every $b\in N$ and $\alpha\in K$, we have
\begin{align*}
W^F(V_b^F(\mathbb{V}))&=\left(\left\{\frac{c_n}{c_{n/b}}W_{n/b}^F(\mathbb{V})\right\}_{b\mid_Nn},\{0\}_{b\nmid_Nn}\right), \\
V_b^F(S^F(\mathbb{V},\mathbb{V}'))&=S^F(V_b^F(\mathbb{V}),V_b^F(\mathbb{V}')), \\
V_{q^i}^F(\mathbb{V})&=[\tau^i]^F(\mathbb{V}), \\
W^F(F_b^F(\mathbb{V}))&=\{W_{bn}^F(\mathbb{V})\}, \\
F_b^F(S^F(\mathbb{V},\mathbb{V}'))&=S^F(F_b^F(\mathbb{V}),F_b^F(\mathbb{V}')), \\
F_b^F(P^F(\mathbb{V},\mathbb{V}'))&=P^F(F_b^F(\mathbb{V}),F_b^F(\mathbb{V}')),
\end{align*}
and
\[
F_{q^i}^F([\alpha]^F(\mathbb{V}))=[\sigma^i\alpha]^F(F_{q^i}^F(\mathbb{V})).
\]

For $b,b'\in N$ and $i\in\mathbb{N}$, we also have
\begin{align*}
V_{b'}^F(V_b^F(\mathbb{V}))&=V_b^F(V_{b'}^F(\mathbb{V}))=V_{bb'}^F(\mathbb{V}), \\
F_{b'}^F(F_b^F(\mathbb{V}))&=F_b^F(F_{b'}^F(\mathbb{V}))=F_{bb'}^F(\mathbb{V}), \\
P^F(\mathbb{V},V_b^F(\mathbb{V}'))&=V_b^F\bigl(P^F(F_b^F(\mathbb{V}),\mathbb{V}')\bigr),
\end{align*}
and
\[
F_{q^i}^F(V_{q^i}^F(\mathbb{V}))=[\Pi_i]^F(\mathbb{V}).
\]
\end{defn}

\subsection{Integrality of the Transformations and the Operations}

\begin{rem}
We have already seen that
\[
E_n^F,~~W_n^F,~~V_{b,n}^F
\]
are polynomials with coefficients in $R$.
\end{rem}

\begin{defn}
For $\alpha\in K$, define
\[
\operatorname{AH}\text{-}\exp_F(\alpha;T)=\exp_F\left(\sum_{i=0}^{\infty}\frac{\sigma^i(\alpha)}{c_{q^i}}T^{q^i}\right)\in TK[[T]].
\]
We call this the \textit{generalized $F$-Artin--Hasse curve}. We also put
\[
\operatorname{AH}\text{-}\exp_F(T)=\operatorname{AH}\text{-}\exp_F(1;T)=\exp_F(\log_{F_q}(T))\in T+T^2K[[T]],
\]
and call it the \textit{$F$-Artin--Hasse curve}.
\end{defn}

\begin{lem}
For every $\alpha\in R$,
\[
\operatorname{AH}\text{-}\exp_F(\alpha;T)\in TR[[T]].
\]
\end{lem}

\begin{proof}
We have
\[
\log_F\bigl(\operatorname{AH}\text{-}\exp_F(\alpha;T)\bigr)=\sum_{i=0}^{\infty}\frac{\sigma^i(\alpha)}{c_{q^i}}T^{q^i}=\sum_{i=0}^{\infty}\frac{\sigma^i(\alpha)}{\Pi_i}T^{q^i}.
\]
Put
\[
h(T)=\sum_{i=0}^{\infty}\frac{\sigma^i(\alpha)}{\Pi_i}T^{q^i}.
\]
Then
\[
h(T)-\frac{1}{\pi}\sigma(h(T^q))=\alpha T\in R[[T]].
\]
Hence, by (iv) of Lemma~\ref{FEL},
\[
\operatorname{AH}\text{-}\exp_F(\alpha;T)=\exp_F(h(T))=\log_F^{-1}(h(T))\in TR[[T]].
\]
\end{proof}

\begin{prop}
We have
\begin{align*}
&\sum_{k\in N}\hspace{-1mm}{{}^F}(-_F)(-S_k^F(\mathbb{V},\mathbb{V}')T^k) \\
&\qquad=\left(\sum_{n\in N}\hspace{-1mm}{{}^F}(-_F)(-V_nT^n)\right)+_F\left(\sum_{n\in N}\hspace{-1mm}{{}^F}(-_F)(-V_n'T^n)\right),
\end{align*}
and
\[
\sum_{n\in N}\hspace{-1mm}{{}^F}(-_F)(-I_n^F(\mathbb{V})T^n)=(-_F)\left(\sum_{n\in N}\hspace{-1mm}{{}^F}(-_F)(-V_nT^n)\right).
\]
In particular, $S_n^F,I_n^F$ have coefficients in $R$.
\end{prop}

\begin{proof}
We have
\begin{align*}
&\Theta^F\log_F\left(\sum_{k\in N}\hspace{-1mm}{{}^F}(-_F)(-S_k^F(\mathbb{V},\mathbb{V}')T^k)\right)=\sum_{n\in N}
W_n^F(S^F(\mathbb{V},\mathbb{V}'))T^n \\
=&\sum_{n\in N}W_n^F(\mathbb{V})T^n+\sum_{m\in N}W_m^F(\mathbb{V}')T^m \\
=&\Theta^F\log_F\left(\left(\sum_{n\in N}\hspace{-1mm}{{}^F}(-_F)(-V_nT^n)\right)+_F\left(\sum_{m\in N}\hspace{-1mm}{{}^F}(-_F)(-V_m'T^m)\right)\right).
\end{align*}
Similarly,
\begin{align*}
&\Theta^F\log_F\left(\sum_{k\in N}\hspace{-1mm}{{}^F}(-_F)(-I_k^F(\mathbb{V})T^k)\right)=\sum_{n\in N}W_n^F(I^F(\mathbb{V}))T^n \\
=&-\sum_{n\in N}W_n^F(\mathbb{V})T^n=\Theta^F\log_F\left((-_F)\left(\sum_{n\in N}\hspace{-1mm}{{}^F}(-_F)(-V_nT^n)\right)\right).
\end{align*}
Since $\Theta^F\log_F$ is injective, the asserted identities follow.

The right-hand sides belong to $R[\mathbb{V},\mathbb{V}'][[T]]$. The integrality of the formal $F$-Euler product representation polynomials therefore implies that $S_n^F$ and $I_n^F$ have coefficients in $R$.
\end{proof}

\begin{prop}
We have
\[
\sum_{n\in N}\hspace{-1mm}{{}^F}(-_F)(-A_n^F(\mathbb{V}_q)T^n)=\sum_{i=0}^{\infty}\hspace{-1mm}{{}^F}\operatorname{AH}\text{-}\exp_F(V_{q,i}T^{q^i}).
\]
In particular, $A_n^F$ has coefficients in $R$.
\end{prop}

\begin{proof}
We have
\begin{align*}
&\Theta^F\log_F\left(\sum_{n\in N}\hspace{-1mm}{{}^F}(-_F)(-A_n^F(\mathbb{V}_q)T^n)\right)=\sum_{n\in N}W_n^F(A^F(\mathbb{V}_q))T^n=\sum_{i=0}^{\infty}W_{q^i}^{F_q}(\mathbb{V}_q)T^{q^i} \\
=&\sum_{i=0}^{\infty}\left(\sum_{j=0}^{i}\frac{c_{q^i}}{c_{q^{i-j}}}V_{q,j}^{q^{i-j}}\right)T^{q^i}=\sum_{j=0}^{\infty}\sum_{i=j}^{\infty}\frac{c_{q^i}}{c_{q^{i-j}}}V_{q,j}^{q^{i-j}}T^{q^i}=\sum_{j=0}^{\infty}\sum_{k=0}^{\infty}\frac{c_{q^{j+k}}}{c_{q^k}}V_{q,j}^{q^k}T^{q^{j+k}} \\
=&\Theta^F\left(\sum_{j=0}^{\infty}\sum_{k=0}^{\infty}\frac{1}{c_{q^k}}(V_{q,j}T^{q^j})^{q^k}\right)=\Theta^F\left(\sum_{j=0}^{\infty}\log_F\exp_F\left(\sum_{k=0}^{\infty}\frac{1}{c_{q^k}}(V_{q,j}T^{q^j})^{q^k}\right)\right) \\
=&\Theta^F\log_F\left(\sum_{j=0}^{\infty}\hspace{-1mm}{{}^F}\operatorname{AH}\text{-}\exp_F(V_{q,j}T^{q^j})\right).
\end{align*}
Since $\Theta^F\log_F$ is injective,
\[
\sum_{n\in N}\hspace{-1mm}{{}^F}(-_F)(-A_n^F(\mathbb{V}_q)T^n)=\sum_{i=0}^{\infty}{}^F\operatorname{AH}\text{-}\exp_F(V_{q,i}T^{q^i}).
\]

The right-hand side belongs to $R[\mathbb{V}_q][[T]]$. Hence the integrality of the formal $F$-Euler product representation polynomials implies that $A_n^F$ has coefficients in $R$.
\end{proof}

\begin{prop}
For $b\in N$, we have
\begin{align*}
&\sum_{n\in N}\hspace{-1mm}{{}^F}(-_F)(-F_{b,n}^F(\mathbb{V})T^n) \\
=&\sum_{n\in N}\hspace{-1mm}{{}^F}\sum_{m\in N_{(q)}}\hspace{-3mm}{{}^F}\operatorname{AH}\text{-}\exp_F\left(\frac{c_{bnm/(b,n)}}{c_{bm/(b,n)}c_{nm/(b,n)}};\left(V_n^{b/(b,n)}T^{n/(b,n)}\right)^m\right).
\end{align*}
In particular, $F_{b,n}^F$ has coefficients in $R$.
\end{prop}

\begin{proof}
We have
\begin{align*}
&\Theta^F\log_F\left(\sum_{n\in N}\hspace{-1mm}{{}^F}(-_F)(-F_{b,n}^F(\mathbb{V})T^n)\right)=\sum_{n\in N}W_n^F(F_b^F(\mathbb{V}))T^n=\sum_{n\in N}W_{bn}^F(\mathbb{V})T^n \\
=&\sum_{n\in N}\sum_{d\mid_Nbn} \frac{c_{bn}}{c_{bn/d}}V_d^{bn/d}T^n=\sum_{d\in N}\sum_{\frac{d}{(b,d)}\mid_Nn}\frac{c_{bn}}{c_{bn/d}}V_d^{bn/d}T^n \\
=&\sum_{d\in N}\sum_{k\in N}\frac{c_{bdk/(b,d)}}{c_{bk/(b,d)}}V_d^{bk/(b,d)}T^{dk/(b,d)} \\
=&\sum_{d\in N}\sum_{m\in N_{(q)}}\sum_{i=0}^{\infty}\frac{c_{bdq^im/(b,d)}}{c_{bq^im/(b,d)}}\left(V_d^{b/(b,d)}T^{d/(b,d)}\right)^{q^im} \\
=&\sum_{d\in N}\sum_{m\in N_{(q)}}\sum_{i=0}^{\infty}c_{dq^im/(b,d)}\frac{c_{bdq^im/(b,d)}}{c_{bq^im/(b,d)}c_{dq^im/(b,d)}}\left(V_d^{b/(b,d)}T^{d/(b,d)}\right)^{q^im} \\
=&\sum_{d\in N}\sum_{m\in N_{(q)}}\sum_{i=0}^{\infty}c_{dq^im/(b,d)}\frac{\sigma^i\left(\frac{c_{bdm/(b,d)}}{c_{bm/(b,d)}c_{dm/(b,d)}}
\right)}{\Pi_i}\left(V_d^{b/(b,d)}T^{d/(b,d)}\right)^{q^im} \\
=&\Theta^F\left(\sum_{d\in N}\sum_{m\in N_{(q)}}\log_F\exp_F\left(\sum_{i=0}^{\infty}\frac{\sigma^i\left(\frac{c_{bdm/(b,d)}}{c_{bm/(b,d)}c_{dm/(b,d)}}\right)}{\Pi_i}\left(V_d^{b/(b,d)}T^{d/(b,d)}\right)^{q^im}\right)\right) \\
=&\Theta^F\log_F\left(\sum_{d\in N}\hspace{-1mm}{{}^F}\sum_{m\in N_{(q)}}\hspace{-3mm}{{}^F}\operatorname{AH}\text{-}\exp_F\left(\frac{c_{bdm/(b,d)}}{c_{bm/(b,d)}c_{dm/(b,d)}};\left(V_d^{b/(b,d)}T^{d/(b,d)}\right)^m\right)\right).
\end{align*}
Since $\Theta^F\log_F$ is injective, the asserted identity follows.

Furthermore,
\[
\frac{c_{bdm/(b,d)}}{c_{bm/(b,d)}c_{dm/(b,d)}}\in R~~~(d\in N,\ m\in N_{(q)}).
\]
Hence the right-hand side belongs to $R[\mathbb{V}][[T]]$, and the integrality of the formal $F$-Euler product representation polynomials implies that $F_{b,n}^F$ has coefficients in $R$.
\end{proof}

\begin{prop}
We have
\begin{align*}
&\sum_{n\in N}\hspace{-1mm}{{}^F}(-_F)(-P_n^F(\mathbb{V},\mathbb{V}')T^n) \\
=&\sum_{m,n\in N}\hspace{-3mm}{{}^F}\sum_{k\in N_{(q)}}\hspace{-3mm}{{}^F}\operatorname{AH}\text{-}\exp_F\left(\frac{c_{mnk/(m,n)}}{c_{mk/(m,n)}c_{nk/(m,n)}};\left(V_m^{n/(m,n)}{V_n'}^{m/(m,n)}T^{mn/(m,n)}\right)^k\right).
\end{align*}
In particular, $P_n^F$ has coefficients in $R$.
\end{prop}

\begin{proof}
We have
\begin{align*}
&\Theta^F\log_F\left(\sum_{n\in N}\hspace{-1mm}{{}^F}(-_F)(-P_n^F(\mathbb{V},\mathbb{V}')T^n)\right)=\sum_{n\in N}W_n^F(\mathbb{V})W_n^F(\mathbb{V}')T^n \\
=&\sum_{n\in N}\left(\sum_{d\mid_Nn}\frac{c_n}{c_{n/d}}V_d^{n/d}\right)\left(\sum_{d'\mid_Nn}\frac{c_n}{c_{n/d'}}{V_{d'}'}^{n/d'}\right)T^n \\
=&\sum_{n\in N}\sum_{d,d'\mid_Nn}c_n\frac{c_n}{c_{n/d}c_{n/d'}}V_d^{n/d}{V_{d'}'}^{n/d'}T^n \\
=&\sum_{d,d'\in N}\sum_{\frac{dd'}{(d,d')}\mid_Nn}c_n\frac{c_n}{c_{n/d}c_{n/d'}}V_d^{n/d}{V_{d'}'}^{n/d'}T^n \\
=&\sum_{d,d'\in N}\sum_{k\in N}c_{dd'k/(d,d')}\frac{c_{dd'k/(d,d')}}{c_{d'k/(d,d')}c_{dk/(d,d')}}\left(V_d^{d'/(d,d')}{V_{d'}'}^{d/(d,d')}T^{dd'/(d,d')}\right)^k \\
&=\sum_{d,d'\in N}\sum_{m\in N_{(q)}}\sum_{i=0}^{\infty}c_{dd'q^im/(d,d')}\frac{c_{dd'q^im/(d,d')}}{c_{d'q^im/(d,d')}c_{dq^im/(d,d')}} \\
&\qquad\qquad\qquad\qquad\qquad\qquad\times\left(V_d^{d'/(d,d')}{V_{d'}'}^{d/(d,d')}T^{dd'/(d,d')}\right)^{q^im} \\
&=\sum_{d,d'\in N}\sum_{m\in N_{(q)}}\sum_{i=0}^{\infty}c_{dd'q^im/(d,d')}\frac{\sigma^i\left(\frac{c_{dd'm/(d,d')}}{c_{d'm/(d,d')}c_{dm/(d,d')}}\right)}{\Pi_i} \\
&\qquad\qquad\qquad\qquad\qquad\qquad\times\left(V_d^{d'/(d,d')}{V_{d'}'}^{d/(d,d')}T^{dd'/(d,d')}\right)^{q^im} \\
&=\Theta^F\left(\sum_{d,d'\in N}\sum_{m\in N_{(q)}}\log_F\exp_F\left(\sum_{i=0}^{\infty}\frac{\sigma^i\left(\frac{c_{dd'm/(d,d')}}{c_{d'm/(d,d')}c_{dm/(d,d')}}\right)}{\Pi_i}\right.\right.\\
&\qquad\qquad\qquad\qquad\qquad\qquad\left.\left.\times\left(V_d^{d'/(d,d')}{V_{d'}'}^{d/(d,d')}T^{dd'/(d,d')}\right)^{q^im}\right)\right) \\
&=\Theta^F\log_F\left(\sum_{d,d'\in N}\hspace{-3mm}{{}^F}\sum_{m\in N_{(q)}}\hspace{-3mm}{{}^F}\operatorname{AH}\text{-}\exp_F\left(\frac{c_{dd'm/(d,d')}}{c_{d'm/(d,d')}c_{dm/(d,d')}};\right.\right. \\
&\qquad\qquad\qquad\qquad\qquad\qquad\qquad\left.\left.\left(V_d^{d'/(d,d')}{V_{d'}'}^{d/(d,d')}T^{dd'/(d,d')}\right)^m\right)\right).
\end{align*}
Since $\Theta^F\log_F$ is injective, the asserted identity follows.

Moreover,
\[
\frac{c_{dd'm/(d,d')}}{c_{d'm/(d,d')}c_{dm/(d,d')}}\in R~~~(d,d'\in N,\ m\in N_{(q)}).
\]
Hence the right-hand side belongs to $R[\mathbb{V},\mathbb{V}'][[T]]$, and the integrality of the formal $F$-Euler product representation polynomials implies that $P_n^F$ has coefficients in $R$.
\end{proof}

\begin{prop}
For every $\alpha\in R$,
\[
\sum_{n\in N}\hspace{-1mm}{{}^F}(-_F)(-[\alpha]_n^F(\mathbb{V})T^n)=\sum_{n\in N}\hspace{-1mm}{{}^F}\sum_{m\in N_{(q)}}\hspace{-3mm}{{}^F}\operatorname{AH}\text{-}\exp_F\left(\frac{\sigma^{v_q(mn)}(\alpha)}{c_m};(V_nT^n)^m\right).
\]
In particular, for every $u\in R[[\tau;\sigma]]$, the polynomial $[u]_n^F$ has coefficients in $R$.
\end{prop}

\begin{proof}
Let $\alpha\in R$. Then
\begin{align*}
&\Theta^F\log_F\left(\sum_{n\in N}\hspace{-1mm}{{}^F}(-_F)(-[\alpha]_n^F(\mathbb{V})T^n)\right)=\sum_{n\in N}W_n^F([\alpha]^F(\mathbb{V}))T^n \\
=&\sum_{n\in N}\operatorname{ghost}\text{-}[\alpha]_n^F(W^F(\mathbb{V}))T^n=\sum_{n\in N}\sigma^{v_q(n)}(\alpha)W_n^F(\mathbb{V})T^n \\
=&\sum_{n\in N}\sigma^{v_q(n)}(\alpha)\sum_{d\mid_Nn}\frac{c_n}{c_{n/d}}V_d^{n/d}T^n=\sum_{d\in N}\sum_{k\in N}\frac{c_{dk}\sigma^{v_q(dk)}(\alpha)}{c_k}V_d^kT^{dk} \\
=&\sum_{d\in N}\sum_{m\in N_{(q)}}\sum_{i=0}^{\infty}c_{dq^im}\frac{\sigma^{v_q(dq^im)}(\alpha)}{c_{q^im}}(V_dT^d)^{q^im} \\
=&\sum_{d\in N}\sum_{m\in N_{(q)}}\sum_{i=0}^{\infty}c_{dq^im}\frac{\sigma^i\left(\frac{\sigma^{v_q(dm)}(\alpha)}{c_m}\right)}{\Pi_i}(V_dT^d)^{q^im} \\
=&\Theta^F\left(\sum_{d\in N}\sum_{m\in N_{(q)}}\log_F\exp_F\left(\sum_{i=0}^{\infty}\frac{\sigma^i\left(\frac{\sigma^{v_q(dm)}(\alpha)}{c_m}\right)}{\Pi_i}(V_dT^d)^{q^im}\right)\right) \\
=&\Theta^F\log_F\left(\sum_{d\in N}\hspace{-1mm}{{}^F}\sum_{m\in N_{(q)}}\hspace{-3mm}{{}^F}\operatorname{AH}\text{-}\exp_F\left(\frac{\sigma^{v_q(dm)}(\alpha)}{c_m};(V_dT^d)^m\right)\right).
\end{align*}
Since $\Theta^F\log_F$ is injective,
\[
\sum_{n\in N}\hspace{-1mm}{{}^F}(-_F)(-[\alpha]_n^F(\mathbb{V})T^n)=\sum_{d\in N}\hspace{-1mm}{{}^F}\sum_{m\in N_{(q)}}\hspace{-3mm}{{}^F}\operatorname{AH}\text{-}\exp_F\left(\frac{\sigma^{v_q(dm)}(\alpha)}{c_m};(V_dT^d)^m\right).
\]

The right-hand side belongs to $R[\mathbb{V}][[T]]$, and hence the integrality of the formal $F$-Euler product representation polynomials implies that $[\alpha]_n^F$ has coefficients in $R$.

Finally, let $\ds u=\sum_{j=0}^{\infty}\alpha_j\tau^j\in R[[\tau;\sigma]]$. Then
\begin{align*}
[u]_n^F(\mathbb{V})&=\sum_{j=0}^{v_q(n)}{}^S[\alpha_j\tau^j]_n^F(\mathbb{V})=\sum_{j=0}^{v_q(n)}{}^S[\tau^j]_n^F([\alpha_j]^F(\mathbb{V})) \\
&=\sum_{j=0}^{v_q(n)}{}^SV_{q^j,n}^F([\alpha_j]^F(\mathbb{V}))\in R[\mathbb{V}].
\end{align*}
Here and throughout, $\ds \sum{}^S$ denotes summation with respect to $S^F$. Thus $[u]_n^F$ has coefficients in $R$.
\end{proof}

\section{Construction of the $F$-big Witt ring}

For $n\in\mathbb{N}_{>0}$, put
\begin{align*}
\overline{n}=\max_{\mid}\{d\in N\mid d\mid n\},\qquad r_n=\frac{n}{\overline{n}},
\end{align*}
and set
\[
\mathcal{R}=\{r_n\mid n\in\mathbb{N}_{>0}\}.
\]
We note that
\[
\mathbb{N}_{>0}=\coprod_{r\in\mathcal{R}}rN.
\]
Indeed, the monus-closedness of $N$ implies that the least common multiple of two elements of $N$ dividing a given positive integer again belongs to $N$. Hence $\overline n$ is well-defined. Moreover,
\[
\overline{r_n}=1~\text{and}~\overline{rn}=n~~~(r\in\mathcal R,\ n\in N).
\]
Thus every positive integer is uniquely written in the form $rn$ with $r\in\mathcal R$ and $n\in N$.

For a subset $Q\subset\mathbb{N}_{>0}$ and $r\in\mathcal{R}$, put
\[
Q_r=\{n\in N\mid rn\in Q\}.
\]
We call $Q$ an \textit{$N$-factor-closed set} if $Q_r\subset N$ is factor-closed for every $r\in\mathcal{R}$, and put
\[
\mathcal{R}_Q=\{r\in\mathcal{R}\mid Q_r\neq\emptyset\}.
\]
Then
\[
\mathcal{R}_Q\subset Q
\]
and
\[
Q=\coprod_{r\in\mathcal{R}_Q}rQ_r.
\]

For an $N$-factor-closed set $Q$ and $n,b\in\mathbb{N}_{>0}$, we write
\begin{align*}
Q(n)&=\{d\in Q\mid d\mid n\},\\
Q/b&=\{n\in\mathbb{N}_{>0}\mid bn\in Q\}.
\end{align*}
Then $Q(n)$ and $Q/b$ are again $N$-factor-closed.

\begin{thm}[Reformulation of Theorem \ref{thm:intro-F-big-Witt}]\label{thm:F-big-Witt}
Let $B$ be an $R$-algebra and let $Q'\subset Q\subset\mathbb{N}_{>0}$ be $N$-factor-closed sets. Then the following assertions hold.
\begin{enumerate}
\item[\textnormal{(i)}]
Put
\[
W_Q^F(B)=B^Q,
\]
and define
\begin{align*}
+_S:W_Q^F(B)\times W_Q^F(B)&\longrightarrow W_Q^F(B);\\
(\{v_n\},\{v_n'\})&\longmapsto\left\{S_{\overline n}^F(v_{r_nd},v_{r_nd}'\mid d\mid_N\overline n)\right\}_n,
\end{align*}
\begin{align*}
\cdot_P:W_Q^F(B)\times W_Q^F(B)&\longrightarrow W_Q^F(B);\\
(\{v_n\},\{v_n'\})&\longmapsto\left\{P_{\overline n}^F(v_{r_nd},v_{r_nd}'\mid d\mid_N\overline n)\right\}_n,
\end{align*}
and
\begin{align*}
[\cdot]^F:R[[\tau;\sigma]]^{\mathrm{op}}&\longrightarrow\operatorname{End}_{\mathrm{Ab}}(W_Q^F(B));\\
u&\longmapsto\left(\{v_n\}\longmapsto\left\{[u]_{\overline n}^F(v_{r_nd}\mid d\mid_N\overline n)\right\}_n\right).
\end{align*}
Then $W_Q^F(B)$ is an $R$-algebra and a right $R[[\tau;\sigma]]$-module whose multiplicative identity is
\[
\left(\{1\}_{n\in\mathcal{R}_Q},\{0\}_{n\in Q\setminus\mathcal{R}_Q}\right).
\]
The additive inverse of $\{v_n\}\in W_Q^F(B)$ is
\[
\left\{I_{\overline n}^F(v_{r_nd}\mid d\mid_N\overline n)\right\}_n.
\]
We call $W_Q^F(B)$ the \textit{$F$-big Witt ring associated with $Q$}.

\item[\textnormal{(ii)}]
The natural projection
\[
\pi_{Q,Q'}^F(B):W_Q^F(B)\longrightarrow W_{Q'}^F(B);~\{v_n\}_{n\in Q}\longmapsto\{v_n\}_{n\in Q'}
\]
is an $R$-algebra homomorphism.

\item[\textnormal{(iii)}]
The assignment
\[
W_Q^F:R\text{-}\mathfrak{Alg}\longrightarrow R\text{-}\mathfrak{Alg}
\]
is a functor, and
\[
\pi_{Q,Q'}^F:W_Q^F\longrightarrow W_{Q'}^F
\]
is a natural transformation.

\item[\textnormal{(iv)}]
The natural map
\[
\varprojlim_{n\in\mathbb{N}_{>0}}\pi_{Q,Q(n)}^F(B):W_Q^F(B)\xrightarrow{\sim}\varprojlim_{n\in\mathbb{N}_{>0}}W_{Q(n)}^F(B)
\]
is an isomorphism of $R$-algebras. In particular,
\[
\left\{\operatorname{Ker}\bigl(\pi_{Q,Q(n)}^F(B)\bigr)\right\}_{n\in\mathbb{N}_{>0}}
\]
forms a directed fundamental system of open ideals on $W_Q^F(B)$.

\item[\textnormal{(v)}]
Let $\{v_n\},\{v_n'\}\in W_Q^F(B)$. If, for every $n\in Q$, either $v_n=0$ or $v_n'=0$, then
\[
\{v_n\}+_S\{v_n'\}=\{v_n+v_n'\}.
\]

\item[\textnormal{(vi)}]
Equip
\[
\operatorname{Ghost}_Q(B)=B^Q
\]
with componentwise addition and multiplication, and define
\begin{align*}
\operatorname{ghost}\text{-}[\cdot]_Q^F:R[[\tau;\sigma]]^{\mathrm{op}}&\longrightarrow\operatorname{End}(\operatorname{Ghost}_Q(B)); \\
u&\longmapsto\left(\{x_n\}\longmapsto\left\{\operatorname{ghost}\text{-}[u]_{\overline n}^F(x_{r_nd}\mid d\mid_N\overline n)\right\}_n\right).
\end{align*}

Put
\[
\mathcal{C}^F(B)=\left\{\sum_{n\in\mathbb{N}_{>0}}\hspace{-2mm}{{}^F}(-_F)(-v_nT^n)~\middle|~v_n\in B\right\}.
\]
Define
\begin{align*}
\cdot_F:\mathcal{C}^F(B)\times\mathcal{C}^F(B)\longrightarrow\mathcal{C}^F(B)
\end{align*}
by
\begin{align*}
&\left(\sum_{n\ge1}\hspace{-1mm}{{}^F}(-_F)(-v_nT^n),\sum_{n\ge1}\hspace{-1mm}{{}^F}(-_F)(-v_n'T^n)\right) \\
&\longmapsto\sum_{r\in\mathcal R}\hspace{-1mm}{{}^F}\sum_{m,n\in N}\hspace{-2mm}{{}^F}\sum_{k\in N_{(q)}}\hspace{-2mm}{{}^F}\operatorname{AH}\text{-}\exp_F\left(\frac{c_{\frac{mn}{(m,n)}k}}{c_{\frac{m}{(m,n)}k}c_{\frac{n}{(m,n)}k}};\right. \\
&\hspace{15em}\left.\left(v_{rm}^{\frac{n}{(m,n)}}{v_{rn}'}^{\frac{m}{(m,n)}}T^{r\frac{mn}{(m,n)}}\right)^k\right).
\end{align*}

For $\ds u=\sum_{j=0}^{\infty}\alpha_j\tau^j\in R[[\tau;\sigma]]$, define
\begin{align*}
\mathcal{C}\text{-}[\cdot]^F:R[[\tau;\sigma]]^{\mathrm{op}}&\longrightarrow\operatorname{End}(\mathcal{C}^F(B))
\end{align*}
by
\begin{align*}
&\mathcal{C}\text{-}[u]^F\left(\sum_{n\ge1}{}^F(-_F)(-v_nT^n)\right) \\
=&\sum_{r\in\mathcal R}\hspace{-1mm}{{}^F}\sum_{j=0}^{\infty}{}^F\sum_{n\in N}\hspace{-1mm}{{}^F}\sum_{m\in N_{(q)}}\hspace{-3mm}{{}^F}\operatorname{AH}\text{-}\exp_F\left(\frac{\sigma^{v_q(mn)}(\alpha_j)}{c_m};(v_{rn}T^{rq^jn})^m\right).
\end{align*}

Furthermore, put
\[
\mathcal{U}_Q^F(B)=\left\{\sum_{n\notin Q}\hspace{-1mm}{{}^F}(-_F)(-v_nT^n)~\middle|~v_n\in B\right\}\subset\mathcal{C}^F(B),
\]
and define
\[
w_Q^F:W_Q^F(B)\longrightarrow\operatorname{Ghost}_Q(B)
\]
by
\[
w_Q^F(\{v_n\})=\left\{W_{\overline n}^F(v_{r_nd}\mid d\mid_N\overline n)\right\}_{n\in Q},
\]
and
\[
e_Q^F:W_Q^F(B)\longrightarrow\mathcal{C}^F(B)/\mathcal{U}_Q^F(B)
\]
by
\[
e_Q^F(\{v_n\})=\left[\sum_{n\in Q}\hspace{-1mm}{{}^F}(-_F)(-v_nT^n)\right].
\]
Finally, define
\[
\widetilde{\ell}_Q^F=w_Q^F\circ(e_Q^F)^{-1}:\mathcal{C}^F(B)/\mathcal{U}_Q^F(B)\longrightarrow\operatorname{Ghost}_Q(B).
\]

Then $\operatorname{Ghost}_Q(B)$ and $\mathcal{C}^F(B)$ are $R$-algebras and right $R[[\tau;\sigma]]$-modules, $\mathcal{U}_Q^F(B)$ is an ideal of $\mathcal{C}^F(B)$, and $w_Q^F$ and $\widetilde{\ell}_Q^F$ are both $R$-algebra homomorphisms and homomorphisms of right $R[[\tau;\sigma]]$-modules. Moreover, $e_Q^F$ is an isomorphism of $R$-algebras and of right $R[[\tau;\sigma]]$-modules, and the diagram
\[
\xymatrix{
W_Q^F(B)\ar[rd]^{w_Q^F} \ar[d]_{e_Q^F} & \\
\mathcal{C}^F(B)/\mathcal{U}_Q^F(B) \ar[r]_{\widetilde{\ell}_Q^F} \ar@{}[ru]|(.34){\circlearrowright} & \operatorname{Ghost}_Q(B)
}
\]
is commutative.

\item[\textnormal{(vii)}]
Let $D(B)$ denote the set of zero divisors of $B$.
Then
\[
\forall n\in Q,~c_{\overline n}1_B\notin D(B)~\Longrightarrow~w_Q^F\ \text{is injective},
\]
and
\[
\forall n\in Q,~c_{\overline n}1_B\in B^\ast~\Longrightarrow~w_Q^F\ \text{is an isomorphism}.
\]

\item[\textnormal{(viii)}]
Define
\begin{align*}
\tau_Q^F:B&\longrightarrow W_Q^F(B);\\
v&\longmapsto\left(\{v^r\}_{r\in\mathcal{R}_Q},\{0\}_{n\in Q\setminus\mathcal{R}_Q}\right),
\end{align*}
\[
\operatorname{ghost}\text{-}\tau_Q^F:B\longrightarrow\operatorname{Ghost}_Q(B);~x\longmapsto\{x^n\}_{n\in Q},
\]
and
\[
\mathcal{C}\text{-}\tau^F:B\longrightarrow\mathcal{C}^F(B);~v\longmapsto \sum_{r\in\mathcal{R}}\hspace{-1mm}{{}^F}(-_F)(-v^rT^r).
\]
Then $\tau_Q^F$, $\operatorname{ghost}\text{-}\tau_Q^F$, and $\mathcal{C}\text{-}\tau^F$ are multiplicative maps, and the diagrams
\[
\xymatrix{
B \ar@{=}[r] \ar[d]_{\tau_Q^F} & B \ar[d]^{\operatorname{ghost}\text{-}\tau_Q^F} \\
W_Q^F(B) \ar[r]_{w_Q^F} & \operatorname{Ghost}_Q(B) \ar@{}[lu]|{\circlearrowright}
}
\]
and
\[
\xymatrix{
B \ar[r]^{\tau_Q^F} \ar@{=}[d] & W_Q^F(B) \ar[d]^{e_Q^F} \\
B \ar[r]_{\overline{\mathcal{C}\text{-}\tau^F}} & \mathcal{C}^F(B)/\mathcal{U}_Q^F(B) \ar@{}[lu]|{\circlearrowright}
}
\]
are commutative.

\item[\textnormal{(ix)}]
For every $v\in B$ and every $\{v_n\}\in W_Q^F(B)$,
\[
\tau_Q^F(v)\cdot_P\{v_n\}=\{v^nv_n\}.
\]

\item[\textnormal{(x)}]
For $b\in\mathbb{N}_{>0}$, define
\begin{align*}
v_{b,Q}^F:W_{Q/b}^F(B)&\longrightarrow W_Q^F(B);\\
\{v_n\}&\longmapsto (\{v_{n/b}\}_{b\mid n\in Q},\{0\}_{b\nmid n\in Q}),
\end{align*}
and
\begin{align*}
\operatorname{ghost}\text{-}v_{b,Q}^F:\operatorname{Ghost}_{Q/b}(B)&\longrightarrow\operatorname{Ghost}_Q(B);\\
\{x_n\}&\longmapsto\left(\left\{\frac{c_{\overline n}}{c_{\overline{n/b}}}x_{n/b}\right\}_{b\mid n},\{0\}_{b\nmid n}\right).
\end{align*}
On the curve side, define
\[
\mathcal{C}\text{-}v_b^F:\mathcal{C}^F(B)\longrightarrow\mathcal{C}^F(B);~f(T)\longmapsto f(T^b).
\]

Define also
\begin{align*}
f_{b,Q}^F:W_Q^F(B)&\longrightarrow W_{Q/b}^F(B);\\
\{v_n\}&\longmapsto\left\{F_{\overline{br_n},\overline n}^F\left(v_{r_{br_n}d}\mid d\mid_N\overline{br_n}\,\overline n\right)\right\}_{n\in Q/b}, \\
\mathcal{C}\text{-}f_{b}^F&=e_{\mathbb{N}_{>0}}^F\circ f_{b,\mathbb{N}_{>0}}^F\circ (e_{\mathbb{N}_{>0}}^F)^{-1}:\mathcal{C}^F(B)\to \mathcal{C}^F(B)
\end{align*}
and
\[
\operatorname{ghost}\text{-}f_{b,Q}^F:\operatorname{Ghost}_Q(B)\longrightarrow\operatorname{Ghost}_{Q/b}(B);~\{x_n\}\longmapsto\{x_{bn}\}_{n\in Q/b}.
\]

Then
\[
v_{b,Q}^F,~\operatorname{ghost}\text{-}v_{b,Q}^F,~\mathcal{C}\text{-}v_b^F
\]
are homomorphisms of additive groups, while
\[
f_{b,Q}^F,~\operatorname{ghost}\text{-}f_{b,Q}^F,~\mathcal{C}\text{-}f_b^F
\]
are ring homomorphisms. Moreover, the diagrams
\[
\xymatrix{
W_{Q/b}^F(B) \ar[r]^{w_{Q/b}^F} \ar[d]_{v_{b,Q}^F} & \operatorname{Ghost}_{Q/b}(B) \ar[d]^{\operatorname{ghost}\text{-}v_{b,Q}^F} \\
W_Q^F(B) \ar[r]_{w_Q^F} & \operatorname{Ghost}_Q(B) \ar@{}[lu]|{\circlearrowright}
}
\]
\[
\xymatrix{
W_{Q/b}^F(B) \ar[r]^{v_{b,Q}^F} \ar[d]_{e_{Q/b}^F} & W_Q^F(B) \ar[d]^{e_Q^F} \\
\mathcal{C}^F(B)/\mathcal{U}_{Q/b}^F(B) \ar[r]_{\overline{\mathcal{C}\text{-}v_b^F}} & \mathcal{C}^F(B)/\mathcal{U}_Q^F(B) \ar@{}[lu]|{\circlearrowright}
}
\]
and
\[
\xymatrix{
W_Q^F(B) \ar[r]^{w_Q^F} \ar[d]_{f_{b,Q}^F} & \operatorname{Ghost}_Q(B) \ar[d]^{\operatorname{ghost}\text{-}f_{b,Q}^F} \\
W_{Q/b}^F(B) \ar[r]_{w_{Q/b}^F} & \operatorname{Ghost}_{Q/b}(B) \ar@{}[lu]|{\circlearrowright}
}
\]
\[
\xymatrix{
W_{Q}^F(B) \ar[r]^{f_{b,Q}^F} \ar[d]_{e_{Q}^F} & W_{Q/b}^F(B) \ar[d]^{e_{Q/b}^F} \\
\mathcal{C}^F(B)/\mathcal{U}_Q^F(B) \ar[r]_{\overline{\mathcal{C}\text{-}f_b^F}} & \mathcal{C}^F(B)/\mathcal{U}_{Q/b}^F(B) \ar@{}[lu]|{\circlearrowright}
}
\]
are commutative.

\item[\textnormal{(xi)}]
For $b,b'\in\mathbb{N}_{>0}$, we have
\[
v_{b',Q}^F\circ v_{b,Q/b'}^F=v_{b,Q}^F\circ v_{b',Q/b}^F=v_{bb',Q}^F,
\]
and
\[
f_{b',Q/b}^F\circ f_{b,Q}^F=f_{b,Q/b'}^F\circ f_{b',Q}^F=f_{bb',Q}^F.
\]
Moreover, for any $\{v_n\}\in W_Q^F(B),~\{v_n'\}\in W_{Q/b}^F(B)$,
\[
\{v_n\}\cdot_Pv_{b,Q}^F(\{v_n'\})=v_{b,Q}^F\left(f_{b,Q}^F(\{v_n\})\cdot_P\{v_n'\}\right).
\]
If $b\in\langle q\rangle\times N_{(p)}$, then
\[
f_{b,Q}^F\circ v_{b,Q}^F=[c_b]_{Q/b}^F.
\]
In particular,
\[
v_{b,Q}^F\bigl(W_{Q/b}^F(B)\bigr)\subset W_Q^F(B)
\]
is an ideal.

\item[\textnormal{(xii)}]
The map
\[
a^F:W_{\langle q\rangle}^{F_q}(B)\longrightarrow W_N^F(B);~\{v_{q,i}\}\longmapsto\left\{A_n^F(v_{q,i}\mid q^i\mid_Nn)\right\}_{n\in N}
\]
is a ring homomorphism, not necessarily preserving the multiplicative identity. Let
\[
\iota_q:\operatorname{Ghost}_{\langle q\rangle}(B)\longrightarrow\operatorname{Ghost}_N(B)
\]
be the natural embedding given by
\[
\{x_{q,i}\}\longmapsto\left(\{x_{q,i}\}_{n=q^i},\{0\}_{n\notin\langle q\rangle}\right).
\]
Then the diagram
\[
\xymatrix{
W_{\langle q\rangle}^{F_q}(B) \ar[r]^{w_{\langle q\rangle}^{F_q}} \ar[d]_{a^F} & \operatorname{Ghost}_{\langle q\rangle}(B) \ar[d]^{\iota_q} \\
W_N^F(B) \ar[r]_{w_N^F} & \operatorname{Ghost}_N(B) \ar@{}[lu]|{\circlearrowright}
}
\]
is commutative. In particular, if $f_{\sigma/F}=1$, then $a^F$ is a section of $\pi_{N,\langle q\rangle}^F$.

\item[\textnormal{(xiii)}]
The map
\[
\prod_{r\in\mathcal{R}}W_N^F(B)\xrightarrow{\sim}W_{\mathbb{N}_{>0}}^F(B);~\{\{v_n^{(r)}\}_n\}_r\longmapsto\{v_{\overline n}^{(r_n)}\}_n
\]
is a ring isomorphism.

\item[\textnormal{(xiv)}]
Suppose that $f_{\sigma/F}=1$. Then the maps
\[
W_N^F(B)\longrightarrow \prod_{m\in N_{(q)}}W_{\langle q\rangle}^{F_q}(B);~\{v_n\}\longmapsto\left\{\pi_{N,\langle q\rangle}^F\left(f_{m,N}^F(\{v_d\})\right)\right\}_m
\]
and
\begin{align*}
\prod_{m\in N_{(q)}}W_{\langle q\rangle}^{F_q}(B)&\longrightarrow W_N^F(B); \\
\{\{v_{q,i}^{(m)}\}_i\}_m&\longmapsto\sum_{m\in N_{(q)}}{}^S\left(v_{m,N}^F[\rho(m)^{-1}]_N^Fa^F\right)
\left(\{v_{q,i}^{(m)}\}_i\right)
\end{align*}
are mutually inverse ring isomorphisms.
\end{enumerate}
\end{thm}

\begin{proof}
All the ring and module identities below may be checked sectorwise. Since every $c_n$ is a unit in $K$, the ghost map becomes an isomorphism after base change from $R$ to $K$ by the triangularity of the ghost polynomials. Hence the required universal polynomial identities may be verified on ghost components over $K$ and then descend to $R$ by the integrality results of the preceding section.

For (i), on each sector $rN$ the operations are precisely those defined by the polynomials $S^F$, $P^F$, $I^F$, and $[u]^F$. Hence the ring axioms and the right $R[[\tau;\sigma]]$-module structure follow from their formal identities. The multiplicative identity on the $r$-sector is concentrated at the coordinate $r$, which gives $\left(\{1\}_{r\in\mathcal R_Q},\{0\}_{n\in Q\setminus\mathcal R_Q}\right)$.

Assertions (ii) and (iii) follow immediately from the fact that $Q_r'$ is factor-closed in $Q_r$ and that all the defining polynomials are universal over $R$. Assertion (iv) follows coordinatewise, since every coordinate of $W_Q^F(B)$ occurs in $W_{Q(n)}^F(B)$ for a sufficiently divisible $n$. The Hausdorffness and completeness follow from the same observation.

Assertion (v) is the sectorwise form of
\[
S^F\left((\{V_n\}_{n\in M},\{0\}_{n\notin M}),(\{0\}_{n\in M},\{V_n\}_{n\notin M})\right)=\mathbb{V}.
\]

For (vi), the formal $F$-Euler product representation gives a bijection
\[
e_Q^F:W_Q^F(B)\xrightarrow{\sim}\mathcal C^F(B)/\mathcal U_Q^F(B).
\]
The formulas defining $\cdot_F$ and $\mathcal C\text{-}[u]^F$ are exactly the images under $e_Q^F$ of the Witt multiplication and the skew action constructed in the preceding section. Thus $e_Q^F$ is an isomorphism of $R$-algebras and of right $R[[\tau;\sigma]]$-modules. The remaining assertions and the commutative diagram follow from $\widetilde\ell_Q^F=w_Q^F\circ(e_Q^F)^{-1}$.

For (vii), the ghost polynomials are triangular:
\[
W_{\overline n}^F=c_{\overline n}V_{\overline n}+(\text{terms involving only proper $N$-divisors of $\overline n$}).
\]
Hence induction on $\overline n$ proves the injectivity assertion, and the same induction gives bijectivity when every $c_{\overline n}1_B$ is a unit.

For (viii), if $n=r_nd$ with $d\in N$, then
\[
W_d^F(v^{r_n},0,\ldots)=v^{r_nd}=v^n.
\]
Hence
\[
w_Q^F(\tau_Q^F(v))=\{v^n\}_{n\in Q}.
\]
The multiplicativity of the three Teichm\"uller maps and the two commutative diagrams follow from the corresponding universal polynomial identities. Similarly, (ix) follows sectorwise from
\[
P^F((V,0,0,\ldots),\mathbb V')=\{V^nV_n'\}.
\]

For (x), a direct substitution into the ghost polynomials gives
\[
w_Q^F\circ v_{b,Q}^F=\operatorname{ghost}\text{-}v_{b,Q}^F\circ w_{Q/b}^F,
\]
and, by the defining property of the Frobenius polynomials,
\[
w_{Q/b}^F\circ f_{b,Q}^F=\operatorname{ghost}\text{-}f_{b,Q}^F\circ w_Q^F.
\]
The assertions concerning addition, multiplication, and the commutative diagrams therefore follow from the corresponding polynomial identities. The formula
\[
\mathcal C\text{-}v_b^F(g)(T)=g(T^b)
\]
is immediate from the $F$-Euler product representation.

Assertion (xi) follows by evaluating both sides on ghost components. In particular, for $b\in\langle q\rangle\times N_{(p)}$,
\[
\frac{c_{\overline{bn}}}{c_{\overline n}}=\sigma^{v_q(n)}(c_b),
\]
and hence
\[
f_{b,Q}^F\circ v_{b,Q}^F=[c_b]_{Q/b}^F.
\]
The projection formula implies that the image of $v_{b,Q}^F$ is an ideal.

For (xii), by the definition of the polynomials $A_n^F$,
\[
w_N^F\circ a^F=\iota_q\circ w_{\langle q\rangle}^{F_q}.
\]
Thus $a^F$ is a ring homomorphism. If $f_{\sigma/F}=1$, then $N=\langle q\rangle\times N_{(q)}$, and restriction to the $\langle q\rangle$-coordinates gives
\[
\pi_{N,\langle q\rangle}^F\circ a^F=\id.
\]

Assertion (xiii) follows directly from the decomposition
\[
\mathbb N_{>0}=\coprod_{r\in\mathcal R}rN,
\]
since all operations are defined independently on each sector.

Finally, assume $f_{\sigma/F}=1$. Then
\[
N=\langle q\rangle\times N_{(q)}.
\]
On ghost components, the image of $a^F$ is supported precisely on the powers of $q$. Hence, for $m,m'\in N_{(q)}$,
\[
\pi_{N,\langle q\rangle}^Ff_{m,N}^Fv_{m',N}^F[\rho(m')^{-1}]_N^Fa^F=\begin{cases}\id & (m=m')\\ 0 & (m\neq m').\end{cases}
\]
Indeed, in the case $m=m'$ we use
\[
f_{m,N}^Fv_{m,N}^F=[c_m]_N^F=[\rho(m)]_N^F.
\]
Therefore the two maps in (xiv) are mutually inverse ring homomorphisms.
\end{proof}

\begin{rem}
If $k=\mathbb Q_p$, $R=\mathbb Z_p$, $\sigma=\id$, $q=p$, and $F=\widehat{\mathbb G}_m$, then $W_{\mathbb{N}_{>0}}^{\widehat{\mathbb{G}}_m}$ is the usual big Witt ring. 

Moreover, $W_{\langle q\rangle}^{F_q}$ is Hazewinkel's ramified $q$-typical Witt ring.

In what follows, when $Q=\mathbb{N}_{>0}$, we simply write $W^F(B),\operatorname{Ghost}(B),w^F,e^F,$
$\widetilde{\ell}^F,\tau^F,v^F,f^F,a^F$ for the corresponding objects and maps.
\end{rem}

\begin{prop}
Let $B$ be an integral domain which is also a $K$-algebra, and let
\[
b\in\langle q\rangle\times N_{(p)}.
\]
For $g\in\mathcal{C}^F(B)$, 
\[
(\mathcal{C}\text{-}f_b^F(g))(T^b)=\mathcal{C}\text{-}\left[\frac{\sigma^{-v_q(b)}(c_b)}{b}\right]^F\left(\sum_{\zeta^b=1}\hspace{-1mm}{{}^F} g(\zeta T)\right).
\]
\end{prop}

\begin{proof}
Let $\Omega$ be an algebraic closure of the fraction field of $B$. For $g(T)\in\mathcal C^F(B)$, put
\[
G(T)=\sum_{\zeta^b=1}\hspace{-1mm}{{}^F} g(\zeta T)\in\mathcal C^F(\Omega).
\]
Writing $\ds \log_F g(T)=\sum_{n\ge1}a_nT^n$, we obtain
\[
\log_F G(T)=\sum_{\zeta^b=1}\log_F g(\zeta T)=b\sum_{n\ge1}a_{bn}T^{bn}.
\]
Since the formal group law $F$ is defined over $B$, every coefficient of $G(T)$ is a symmetric polynomial over $B$ in the $b$-th roots of
unity. Hence $G(T)\in B[[T]]$. Moreover, for every $\xi$ with $\xi^b=1$, we have $G(\xi T)=G(T)$, and therefore $G(T)\in B[[T^b]]$. Hence
\[
G(T)\in\mathcal C^F(B)\cap B[[T^b]].
\]
Thus the right-hand side in the assertion is a series in $T^b$, and there exists a unique $h\in\mathcal C^F(B)$ such that
\[
h(T^b)=\mathcal{C}\text{-}\left[\frac{\sigma^{-v_q(b)}(c_b)}{b}\right]^F\left(\sum_{\zeta^b=1}\hspace{-1mm}{{}^F} g(\zeta T)\right).
\]

Let $x_n$ be the $n$-th ghost component of $g$. The $n$-th ghost component of $G$ is
\[
\begin{cases}
bx_n & (b\mid n), \\ 0 & (b\nmid n).
\end{cases}
\]
Since the ghost action of $\mathcal C\text{-}[\alpha]^F$ is given by
\[
x_n\longmapsto\sigma^{v_q(n)}(\alpha)x_n,
\]
the $bt$-th ghost component of $\ds \mathcal C\text{-}\left[\frac{\sigma^{-v_q(b)}(c_b)}{b}\right]^FG$ is
\[
\sigma^{v_q(bt)}\left(\frac{\sigma^{-v_q(b)}(c_b)}{b}\right)b\,x_{bt}=\sigma^{v_q(t)}(c_b)x_{bt}.
\]
On the other hand, since $b\in\langle q\rangle\times N_{(p)}$,
\[
\frac{c_{\overline{bt}}}{c_{\overline t}}=\sigma^{v_q(t)}(c_b).
\]
Therefore, if $y_t$ denotes the $t$-th ghost component of $h$, then the relation
\[
h(T^b)=\mathcal C\text{-}\left[\frac{\sigma^{-v_q(b)}(c_b)}{b}\right]^FG(T)
\]
gives $\ds \frac{c_{\overline{bt}}}{c_{\overline t}}y_t=\frac{c_{\overline{bt}}}{c_{\overline t}}x_{bt}$, and hence $y_t=x_{bt}$. Thus
\[
\widetilde{\ell}^F(h)=\operatorname{ghost}\text{-}f_b^F\bigl(\widetilde{\ell}^F(g)\bigr).
\]
On the other hand, by the definition
\[
\mathcal C\text{-}f_b^F=e^F\circ f_b^F\circ(e^F)^{-1},
\]
we also have
\[
\widetilde{\ell}^F\bigl(\mathcal C\text{-}f_b^F(g)\bigr)=\operatorname{ghost}\text{-}f_b^F\bigl(\widetilde{\ell}^F(g)\bigr).
\]
Since $B$ is a $K$-algebra, all $c_{\overline n}$ are units in $B$, and hence $\widetilde{\ell}^F$ is an isomorphism. Therefore
\[
h=\mathcal C\text{-}f_b^F(g).
\]
This proves the asserted formula.
\end{proof}

\section{Generalized Fleck--Sun--Wan congruences}

Let $F$ be a Lubin--Tate formal $\ded_k$-module, and put
\[
q=\#\kappa_k
\]
We normalize the valuation $v_k$ by $v_k(\pi)=1$. Put
\begin{align*}
\mathcal{R}_k&=\left\{\sum_{n\in\mathbb{Z}}a_nT^n~\middle|~a_n\in\ded_k,~\lim_{n\to\infty}v_k(a_{-n})=\infty\right\}, \\
\mathcal{U}_k&=\mathcal{R}_k^\ast=\left\{\sum_{n\in\mathbb{Z}}a_nT^n\in\mathcal{R}_k~\middle|~a_n\in \ded_k^\ast\text{ for some }n\in\mathbb{Z}\right\}, \\
\mathcal{Q}_k&=\mathcal{R}_k\left[\frac{1}{\pi}\right] \\
&=\left\{\sum_{n\in\mathbb{Z}}a_nT^n~\middle|~\begin{array}{l} \ds a_n\in k,~\lim_{n\to\infty}v_k(a_{-n})=\infty, \\ \{v_k(a_n)\}_{n\in\mathbb Z}\text{ is bounded below}\end{array}\right\}, \\
F[\pi^s]&=\{a\in\overline k\mid[\pi^s]_F(a)=0\}.
\end{align*}
Define $\varphi_k:\mathcal{Q}_k\to \mathcal{Q}_k$ by
\[
(\varphi_k(g))(T)=g([\pi]_F(T)).
\]
Define the Coleman trace and the Coleman norm by
\[
\text{Tr}_{Co}:\mathcal{Q}_k\longrightarrow\mathcal{Q}_k;~g\longmapsto h,
\]
where $h$ is characterized by
\[
h([\pi]_F(T))=\frac{1}{\pi}\sum_{a\in F[\pi]}g(a+_FT),
\]
and
\[
N_{Co}:\mathcal{Q}_k\longrightarrow\mathcal{Q}_k;~g\longmapsto h,
\]
where $h$ is characterized by
\[
h([\pi]_F(T))=\prod_{a\in F[\pi]}g(a+_FT).
\]
Put
\[
\partial_F=\frac{d}{d\log_F(T)}=\frac{1}{\log_F'(T)}\frac{d}{dT}.
\]

\begin{defn}
Define the polynomials $\ds \binom{X}{m}_F$ by
\[
\exp(X\log_F(U))=\sum_{m=0}^{\infty}\binom{X}{m}_F U^m.
\]
For $s\in\mathbb{N}$ and $m\in\mathbb{Z}$, define
\[
\mathfrak{F}_{s,m}^F:\mathcal{Q}_k\longrightarrow k;~g\longmapsto (\text{the~coefficient~of~}T^m\text{~in~}\text{Tr}_{Co}^s(g)).
\]

We call $\ds \binom{X}{m}_F$ the \textit{$F$-binomial polynomial}, and $\mathfrak{F}_{s,m}^F$ the \textit{Fleck operator}.

By definition,
\begin{align*}
g(T+_FU)&=g\bigl(\exp_F(\log_F(T)+\log_F(U))\bigr) \\
&=\sum_{r=0}^{\infty}\frac{1}{r!}\frac{d^r(g\circ\exp_F)}{d\log_F(U)^r}(\log_F(T))\log_F(U)^r \\
&=\sum_{r=0}^{\infty}\frac{1}{r!}(\partial_F^rg)(T)\log_F(U)^r \\
&=\left(\left(\sum_{r=0}^{\infty}\frac{1}{r!}(\log_F(U)\partial_F)^r\right)(g)\right)(T) \\
&=\bigl((\exp(\log_F(U)\partial_F))g\bigr)(T) \\
&=\sum_{m=0}^{\infty}\left(\binom{\partial_F}{m}_F g\right)(T)U^m.
\end{align*}

In particular, when $k=\mathbb{Q}_p$ and $F=\widehat{\mathbb{G}}_m$,
\begin{align*}
\partial_{\widehat{\mathbb{G}}_m}&=(1+T)\frac{d}{dT}, \\
\binom{X}{m}_{\widehat{\mathbb{G}}_m}&=\binom{X}{m}.
\end{align*}
\end{defn}

\begin{prop}
For $n\in\mathbb{N}$, we have
\[
\mathfrak{F}_{s,m}^F(T^n)=\begin{cases} \ds \frac{1}{\pi^s}\sum_{a\in F[\pi^s]}\left(\binom{\partial_F/\pi^s}{m}_F T^n\right)(a) & (m\geq 0),
\\[10pt] 0 & (m<0), \end{cases}
\]
and
\[
\mathfrak{F}_{s,m}^F(T^{-n})=\begin{cases} \ds \frac{1}{\pi^s}\sum_{a\in F[\pi^s]}\left(\binom{\partial_F/\pi^s}{m+n}_F\left(\frac{[\pi^s]_F(T)}{T}\right)^n\right)(a) & (m\geq -n), \\[10pt] 0 & (m<-n).\end{cases}
\]
\end{prop}

\begin{proof}
First,
\begin{align*}
&\sum_{m=0}^{\infty}\mathfrak{F}_{s,m}^F(T^n)([\pi^s]_F(U))^m=\bigl(\text{Tr}_{Co}^s(T^n)\bigr)([\pi^s]_F(U)) \\
=&\frac{1}{\pi^s}\sum_{a\in F[\pi^s]}(a+_FU)^n=\frac{1}{\pi^s}\sum_{a\in F[\pi^s]}\bigl((\exp(\log_F(U)\partial_F))T^n\bigr)(a) \\
=&\frac{1}{\pi^s}\sum_{a\in F[\pi^s]}\bigl((\exp(\pi^{-s}\log_F([\pi^s]_F(U))\partial_F))T^n\bigr)(a) \\
=&\frac{1}{\pi^s}\sum_{a\in F[\pi^s]}\sum_{m=0}^{\infty}\left(\binom{\partial_F/\pi^s}{m}_F T^n\right)(a)([\pi^s]_F(U))^m.
\end{align*}
Comparing coefficients of $([\pi^s]_F(U))^m$, we obtain
\[
\mathfrak{F}_{s,m}^F(T^n)=\begin{cases} \ds \frac{1}{\pi^s}\sum_{a\in F[\pi^s]}\left(\binom{\partial_F/\pi^s}{m}_F T^n\right)(a) & (m\geq 0), \\[10pt]0 & (m<0).\end{cases}
\]

On the other hand,
\begin{align*}
\text{Tr}_{Co}^s(T^{-n})&=\text{Tr}_{Co}^s\left([\pi^s]_F(T)^{-n}\left(\frac{[\pi^s]_F(T)}{T}\right)^n\right) \\
&=T^{-n}\text{Tr}_{Co}^s\left(\left(\frac{[\pi^s]_F(T)}{T}\right)^n\right) \\
&=T^{-n}\frac{1}{\pi^s}\sum_{a\in F[\pi^s]}\sum_{m=0}^{\infty}\left(\binom{\partial_F/\pi^s}{m}_F\left(\frac{[\pi^s]_F(T)}{T}\right)^n\right)(a)T^m \\
&=\frac{1}{\pi^s}\sum_{a\in F[\pi^s]}\sum_{m=-n}^{\infty}\left(\binom{\partial_F/\pi^s}{m+n}_F\left(\frac{[\pi^s]_F(T)}{T}\right)^n\right)(a)T^m.
\end{align*}
Hence
\[
\mathfrak{F}_{s,m}^F(T^{-n})=\begin{cases} \ds \frac{1}{\pi^s}\sum_{a\in F[\pi^s]}\left(\binom{\partial_F/\pi^s}{m+n}_F\left(\frac{[\pi^s]_F(T)}{T}\right)^n\right)(a) & (m\geq -n), \\[10pt] 0 & (m<-n).\end{cases}
\]
\end{proof}

\begin{cor}
For $n\in\mathbb{N}$, the Fleck operators associated with
$\widehat{\mathbb{G}}_m$ are given by
\[
\mathfrak{F}_{s,m}^{\widehat{\mathbb{G}}_m}(T^n)=(-1)^n\sum_{\substack{0\leq j\leq n\\p^s\mid j}}(-1)^j\binom{n}{j}\binom{j/p^s}{m}~~~(m\geq 0),
\]
and
\begin{align*}
&\mathfrak{F}_{s,m}^{\widehat{\mathbb{G}}_m}(T^{-n}) \\
=&\sum_{\substack{i_0,\dots,i_{p^s-1}\geq0\\i_0+\cdots+i_{p^s-1}=n\\p^s\mid i_1+2i_2+\cdots+(p^s-1)i_{p^s-1}}}\binom{n}{i_0,\dots,i_{p^s-1}}\binom{(i_1+2i_2+\cdots+(p^s-1)i_{p^s-1})/p^s}{m+n} \\
&\qquad\qquad\qquad\qquad\qquad\qquad\qquad\qquad\qquad\qquad\qquad\qquad\qquad\qquad (m\geq -n).
\end{align*}
\end{cor}

\begin{proof}
Let $\zeta$ be a primitive $p^s$-th root of unity. Since
\[
\widehat{\mathbb G}_m[p^s]=\{\zeta^a-1\mid 0\le a<p^s\}
\]
and
\[
[p^s]_{\widehat{\mathbb G}_m}(T)=(1+T)^{p^s}-1,
\]
the definition of the Coleman trace gives
\[
(\operatorname{Tr}_{Co}^s g)\bigl((1+T)^{p^s}-1\bigr)=\frac{1}{p^s}\sum_{a=0}^{p^s-1}g\bigl(\zeta^a(1+T)-1\bigr).
\]
Put $U=(1+T)^{p^s}-1$. First, for $n\in\mathbb N$, we have
\begin{align*}
(\operatorname{Tr}_{Co}^s T^n)(U)&=\frac{1}{p^s}\sum_{a=0}^{p^s-1}\bigl(\zeta^a(1+T)-1\bigr)^n \\
&=\frac{1}{p^s}\sum_{a=0}^{p^s-1}\sum_{j=0}^n(-1)^{n-j}\binom{n}{j}\zeta^{aj}(1+T)^j \\
&=\sum_{\substack{0\le j\le n\\p^s\mid j}}(-1)^{n-j}\binom{n}{j}(1+T)^j \\
&=(-1)^n\sum_{\substack{0\le j\le n\\p^s\mid j}}(-1)^j\binom{n}{j}(1+U)^{\frac{j}{p^s}} \\
&=\sum_{m=0}^{\left\lfloor\frac{n}{p^s}\right\rfloor}\left((-1)^n\sum_{\substack{0\le j\le n\\p^s\mid j}}(-1)^j\binom{n}{j}\binom{j/p^s}{m}\right)U^m,
\end{align*}
where we used
\[
\frac{1}{p^s}\sum_{a=0}^{p^s-1}\zeta^{aj}=\begin{cases} 1 & (p^s\mid j), \\ 0 & (p^s\nmid j).\end{cases}
\]
Therefore, for $m\ge 0$,
\[
\mathfrak F_{s,m}^{\widehat{\mathbb G}_m}(T^n)=(m\operatorname{-th~coefficient~of~}\operatorname{Tr}_{Co}^sT^n)=(-1)^n\sum_{\substack{0\le j\le n\\p^s\mid j}}(-1)^j\binom{n}{j}\binom{j/p^s}{m}.
\]

Next, consider $T^{-n}$. For every $0\le a<p^s$,
\[
\frac{(1+T)^{p^s}-1}{\zeta^a(1+T)-1}=\frac{(\zeta^a(1+T))^{p^s}-1}{\zeta^a(1+T)-1}=\sum_{j=0}^{p^s-1}(\zeta^a(1+T))^j.
\]
Hence
\[
(\zeta^a(1+T)-1)^{-n}=U^{-n}\left(\sum_{j=0}^{p^s-1}\zeta^{aj}(1+T)^j\right)^n.
\]
It follows that
\begin{align*}
(\operatorname{Tr}_{Co}^sT^{-n})(U)&=\frac{U^{-n}}{p^s}\sum_{a=0}^{p^s-1}\left(\sum_{j=0}^{p^s-1}\zeta^{aj}(1+T)^j\right)^n\\
&=U^{-n}\sum_{\substack{i_0,\ldots,i_{p^s-1}\ge 0 \\ i_0+\cdots+i_{p^s-1}=n}}\binom{n}{i_0,\ldots,i_{p^s-1}}(1+T)^{i_1+2i_2+\cdots+(p^s-1)i_{p^s-1}}\\
&\qquad\qquad\times\frac{1}{p^s}\sum_{a=0}^{p^s-1}\zeta^{a(i_1+2i_2+\cdots+(p^s-1)i_{p^s-1})}.
\end{align*}
Applying the roots-of-unity filter once again, we obtain
\begin{align*}
(\operatorname{Tr}_{Co}^sT^{-n})(U)&=U^{-n}\sum_{\substack{i_0,\ldots,i_{p^s-1}\ge 0 \\ i_0+\cdots+i_{p^s-1}=n \\ p^s\mid i_1+2i_2+\cdots+(p^s-1)i_{p^s-1}}}\binom{n}{i_0,\ldots,i_{p^s-1}}\\
&\qquad\qquad\times (1+U)^{\frac{i_1+2i_2+\cdots+(p^s-1)i_{p^s-1}}{p^s}}. \\
&=\sum_{m=-n}^{\left\lfloor\frac{(p^s-1)n}{p^s}\right\rfloor-n}\Bigg(\sum_{\substack{i_0,\dots,i_{p^s-1}\ge 0 \\ i_0+\cdots+i_{p^s-1}=n \\ p^s\mid i_1+2i_2+\cdots+(p^s-1)i_{p^s-1}}}\binom{n}{i_0,\dots,i_{p^s-1}}\\
&\qquad\qquad\times\binom{(i_1+2i_2+\cdots+(p^s-1)i_{p^s-1})/p^s}{m+n}\Bigg)U^m
\end{align*}
Consequently, for $m\ge -n$,
\begin{align*}
\mathfrak F_{s,m}^{\widehat{\mathbb G}_m}(T^{-n})&=(m\operatorname{-th~coefficient~of~}\operatorname{Tr}_{Co}^sT^{-n}) \\
&=\sum_{\substack{i_0,\dots,i_{p^s-1}\ge 0 \\ i_0+\cdots+i_{p^s-1}=n \\ p^s\mid i_1+2i_2+\cdots+(p^s-1)i_{p^s-1}}}\binom{n}{i_0,\dots,i_{p^s-1}}\\
&\qquad\qquad\times\binom{(i_1+2i_2+\cdots+(p^s-1)i_{p^s-1})/p^s}{m+n}.
\end{align*}
This proves the assertion.
\end{proof}

\begin{proof}[Proof of Theorem \ref{general-FSW}]
Since
\[
\operatorname{Tr}_{Co}\bigl(\ded_k[[T]]\bigr)\subset \ded_k[[T]],
\]
the inclusion
\[
\operatorname{Tr}_{Co}(T^ng(T))\in\sum_{j=0}^{\lfloor n/q\rfloor}T^j\pi^{\left\lfloor\frac{n-1-jq}{q-1}\right\rfloor}\ded_k[[T]]
\]
is immediate for $0\le n<q$. Now let $n\ge q$. Since
\[
[\pi]_F(T)\equiv T^q\pmod{\pi},
\]
we have
\begin{align*}
\operatorname{Tr}_{Co}(T^ng(T))&=\operatorname{Tr}_{Co}\left([\pi]_F(T)T^{n-q}g(T)-([\pi]_F(T)-T^q)T^{n-q}g(T)\right)\\
&=T\operatorname{Tr}_{Co}(T^{n-q}g(T))-\pi\operatorname{Tr}_{Co}\left(T^{n-q+1}\frac{[\pi]_F(T)-T^q}{\pi T}g(T)\right).
\end{align*}
It follows inductively that
\[
\text{Tr}_{Co}(T^ng(T))\in\sum_{j=0}^{\left\lfloor n/q\right\rfloor}T^j\pi^{\left\lfloor\frac{n-1-jq}{q-1}\right\rfloor}\ded_k[[T]].
\]
Iterating this argument, we obtain
\[
\text{Tr}_{Co}^s(T^ng(T))\in\sum_{j=0}^{\left\lfloor n/q^s\right\rfloor}T^j\pi^{\left\lfloor\frac{n-q^{s-1}-jq^s}{(q-1)q^{s-1}}\right\rfloor}\ded_k[[T]].
\]
Since $\ds \left\lfloor\frac{n-q^{s-1}-jq^s}{(q-1)q^{s-1}}\right\rfloor$ is decreasing as a function of $j$, it follows that
\[
v_k\left(\mathfrak{F}_{s,m}^F(T^n)\right)\geq\left\lfloor\frac{n-q^{s-1}-mq^s}{(q-1)q^{s-1}}\right\rfloor.
\]
On the other hand, since $\operatorname{Tr}_{Co}^s(T^n)\in\ded_k[[T]]$, we also have
\[
v_k\left(\mathfrak F_{s,m}^F(T^n)\right)\ge 0.
\]
Therefore,
\[
v_k\left(\mathfrak F_{s,m}^F(T^n)\right)\ge\left\lfloor\frac{n-q^{s-1}-mq^s}{(q-1)q^{s-1}}\right\rfloor_+.
\]
This proves part (i).

Next, since $\ds \frac{[\pi]_F(T)}{T}\in(\pi,T^{q-1})\ded_k[[T]]$, we have
\[
\left(\frac{[\pi]_F(T)}{T}\right)^n\in\sum_{i=0}^n\pi^{n-i}T^{(q-1)i}\ded_k[[T]].
\]
Therefore,
\begin{align*}
\text{Tr}_{Co}(T^{-n}g(T))&=\text{Tr}_{Co}\left([\pi]_F(T)^{-n}\left(\frac{[\pi]_F(T)}{T}\right)^ng(T)\right)\\
&=T^{-n}\text{Tr}_{Co}\left(\left(\frac{[\pi]_F(T)}{T}\right)^ng(T)\right) \\
&\in T^{-n}\text{Tr}_{Co}\left(\sum_{i=0}^n\pi^{n-i}T^{(q-1)i}\ded_k[[T]]\right) \\
&\subset T^{-n}\sum_{i=0}^n\pi^{n-i}\sum_{j=0}^{\left\lfloor (q-1)i/q\right\rfloor}T^j\pi^{\left\lfloor\frac{(q-1)i-1-jq}{q-1}\right\rfloor}\ded_k[[T]] \\
&=\sum_{i=0}^n\sum_{j=0}^{\left\lfloor (q-1)i/q\right\rfloor}T^{-(n-j)}\pi^{n-i+\left\lfloor\frac{(q-1)i-1-jq}{q-1}\right\rfloor}\ded_k[[T]] \\
&\subset \sum_{j=0}^{\left\lfloor (q-1)n/q\right\rfloor}T^{-(n-j)}\pi^{\left\lfloor\frac{(q-1)n-1-jq}{q-1}\right\rfloor}\ded_k[[T]].
\end{align*}
Iterating this inclusion gives
\[
\text{Tr}_{Co}^s(T^{-n}g(T))\in\sum_{j=0}^{\left\lfloor\frac{(q-1)(sn-s+1)}{sq-s+1}\right\rfloor}T^{-(n-j)}\pi^{\left\lfloor\frac{(q-1)(sn-s+1)-1-j(sq-s+1)}{q-1}\right\rfloor}\ded_k[[T]].
\]
Since $\ds \left\lfloor\frac{(q-1)(sn-s+1)-1-j(sq-s+1)}{q-1}\right\rfloor$ is decreasing as a function of $j$, we obtain
\[
v_k\left(\mathfrak{F}_{s,m}^F(T^{-n})\right)\geq\left\lfloor\frac{(q-1)(1-s-sm)-m-n-1}{q-1}\right\rfloor.
\]
On the other hand, since $\operatorname{Tr}_{Co}^s(T^{-n})\in\mathcal R_k$, all its coefficients belong to $\ded_k$. Hence
\[
v_k\left(\mathfrak F_{s,m}^F(T^{-n})\right)\ge0.
\]
Therefore,
\[
v_k\left(\mathfrak F_{s,m}^F(T^{-n})\right)\ge\left\lfloor\frac{(q-1)(1-s-sm)-m-n-1}{q-1}\right\rfloor_+.
\]
This proves part (ii), and hence the theorem.
\end{proof}

\section{The $F$-big Witt ring and Fleck operators}

When the Coleman trace before normalization, with its constant term removed, is lifted to the $F$-big Witt ring, the Fleck operators appear in its Cartier expansion.

The integrality of the Coleman $F$-norm gives rise to congruences for these Fleck-type coefficients. Together with their valuation estimates, these congruences allow the resulting Cartier expansion to be extended to an arbitrary $\pi$-adically complete Hausdorff $\ded_k$-algebra $B$.

From now on, we return to the coefficient-multiplicative setting of the preceding sections and specialize to
\[
R=\ded_k,\qquad q=\#\kappa_k,\qquad \sigma=\id,
\]
and we use the notation
\[
f_k:=f_\sigma,~f_{k/F}:=f_{\sigma/F}.
\]

\begin{defn}
By the general theory of Lubin--Tate formal groups, for every $g\in\mathcal C^F(\ded_k)$ we have
\[
\sum_{a\in F[\pi]}\hspace{-2mm}{{}^F}\bigl(g(a+_FT)-_Fg(a)\bigr)\in\mathcal C^F(\ded_k)^{F[\pi]}=\varphi_k\mathcal C^F(\ded_k).
\]
Hence we define
\[
\mathcal C\text{-}N_{Co}^F:\mathcal C^F(\ded_k)\longrightarrow\mathcal C^F(\ded_k);~g(T)\longmapsto h(T),
\]
where $h$ is characterized by
\[
h([\pi]_F(T))=\sum_{a\in F[\pi]}\hspace{-2mm}{{}^F}\bigl(g(a+_FT)-_Fg(a)\bigr).
\]
We call $\mathcal C\text{-}N_{Co}^F$ the \textit{Coleman $F$-norm on Lubin--Tate curves}.

Since
\[
\log_F:\mathcal C^F(\ded_k)\xrightarrow{\sim}\mathcal{LT}_{\ded_k}
\]
is a bijection by Corollary \ref{FEL-cor}, the operator
\begin{align*}
\pi\widetilde{\operatorname{Tr}}_{Co}&=\log_F\circ\mathcal C\text{-}N_{Co}^F\circ\exp_F:\mathcal{LT}_{\ded_k}\longrightarrow \mathcal{LT}_{\ded_k}; \\
g&\longmapsto \sum_{a\in F[\pi]}\bigl(g(a+_FT)-g(a)\bigr).
\end{align*}
can be defined, and the diagram
\[
\xymatrix{
\mathcal C^F(\ded_k) \ar[r]^{\mathcal C\text{-}N_{Co}^F} \ar[d]_{\log_F} & \mathcal C^F(\ded_k) \ar[d]^{\log_F} \\
\mathcal{LT}_{\ded_k} \ar[r]_{\pi\widetilde{\text{Tr}}_{Co}} & \mathcal{LT}_{\ded_k} \ar@{}[lu]|{\circlearrowright}
}
\]
is commutative. For $g\in\mathcal{LT}_{\ded_k}\cap\mathcal Q_k$, this extension agrees with the normalized Coleman trace, namely
\[
\pi\widetilde{\operatorname{Tr}}_{Co}(g)=\pi\bigl(\operatorname{Tr}_{Co}(g)-(\operatorname{Tr}_{Co}(g))(0)\bigr).
\]

When $k=\mathbb Q_p$ and $F=\widehat{\mathbb G}_m$, for every $g\in T\mathbb Z_p[[T]]$ we have
\[
1+\mathcal C\text{-}N_{Co}^{\widehat{\mathbb G}_m}(g)=\frac{N_{Co}(1+g)}{(N_{Co}(1+g))(0)}.
\]
\end{defn}

\begin{lem}\label{lem:eta-F}
Let $\gamma:\mathbb N_{(p)}\longrightarrow\ded_k$ be the Dirichlet inverse of $\rho(\overline{\cdot})^{-1}$, that is,
\[
\sum_{d\mid m}\gamma(d)\frac{1}{\rho(\overline{m/d})}=\delta_{m,1},
\]
where $\delta_{m,1}$ denotes the Kronecker delta. Define $\{\beta_i\}_{i\in\mathbb N}$ by
\[
\sum_{i=0}^{\infty}\beta_iT^i=\frac{1-T}{1-T^{f_F}}\left(1-\frac{T^{f_k}}{\pi}\right)\left(\sum_{r=0}^{f_{k/F}-1}\nu_r^{-1}T^{f_Fr}\right)^{-1},
\]
and define 
\[
\eta^F:\mathbb{N}_{>0}\longrightarrow k;~n=p^im\longmapsto \beta_i\gamma(m)~~~(i\in\mathbb{N},~m\in\mathbb{N}_{(p)}).
\]
Then $\eta^F$ is the Dirichlet inverse of $c_{\overline{\cdot}}^{-1}$, that is,
\[
\sum_{d\mid n}\frac{\eta^F(d)}{c_{\overline{n/d}}}=\delta_{n,1}~~~(n\in\mathbb N_{>0})
\]
In particular,
\[
v_k(\eta^F(n))\ge -1~~~(n\in\mathbb N_{>0}).
\]
If moreover $N_{(p)}=\mathbb N_{(p)}$, then
\[
\gamma(m)=\frac{\mu(m)}{\rho(m)}.
\]
\end{lem}

\begin{proof}
Since
\[
\sum_{i=0}^{\infty} \frac{1}{c_{\overline{p^i}}}T^i=\frac{1-T^{f_F}}{1-T}\left(1-\frac{T^{f_k}}{\pi}\right)^{-1}\left(\sum_{r=0}^{f_{k/F}-1} \nu_r^{-1}T^{f_Fr}\right),
\]
then
\[
\sum_{i=0}^{\infty} \left(\sum_{j=0}^i \frac{\beta_j}{c_{\overline{p^{i-j}}}}\right)T^i=\left(\sum_{i=0}^{\infty} \beta_iT^i\right)\left(\sum_{i=0}^{\infty} \frac{1}{c_{\overline{p^i}}}T^i\right)=1,
\]
and then
\[
\sum_{j=0}^i \frac{\beta_j}{c_{\overline{p^{i-j}}}}=\delta_{i,0}~~~(\forall i\in\mathbb{N}).
\]
Since $\sigma=\id$, for any $n=p^im\in\mathbb{N}_{>0}~~~(i\in\mathbb{N},~m\in\mathbb{N}_{(p)})$, we obtain
\[
c_{\overline{p^im}}=c_{\overline{p^i}}c_{\overline{m}}=c_{\overline{p^i}}\rho(\overline m).
\]
Therefore
\begin{align*}
\sum_{d\mid n}\frac{\eta^F(d)}{c_{\overline{n/d}}}&=\sum_{j=0}^{i}\sum_{d'\mid m}\frac{\beta_j\gamma(d')}{c_{\overline{p^{i-j}(m/d')}}} \\
&=\sum_{j=0}^{i}\sum_{d'\mid m}\frac{\beta_j\gamma(d')}{c_{\overline{p^{i-j}}}\rho(\overline{m/d'})} \\
&=\left(\sum_{j=0}^{i}\frac{\beta_j}{c_{\overline{p^{i-j}}}}\right)\left(\sum_{d'\mid m}\frac{\gamma(d')}{\rho(\overline{m/d'})}\right) \\
&=\delta_{i,0}\delta_{m,1}=\delta_{n,1}.
\end{align*}
Hence $\eta^F$ is the Dirichlet inverse of $c_{\overline{\cdot}}^{-1}$.

Moreover,
\[
\sum_{i=0}^{\infty}\beta_iT^i=\frac{1-T}{1-T^{f_F}}\left(1-\frac{T^{f_k}}{\pi}\right)\left(\sum_{r=0}^{f_{k/F}-1}\nu_r^{-1}T^{f_Fr}\right)^{-1}\in\pi^{-1}\ded_k[[T]],
\]
so that
\[
v_k(\beta_i)\ge -1~~~(i\in\mathbb N).
\]
Since $\rho(\overline m)\in\ded_k^\ast$, the Dirichlet inverse $\gamma$ takes values in $\ded_k$: indeed, $\gamma(1)=1$, and for
$m>1$,
\[
\gamma(m)=-\sum_{\substack{d\mid m\\d<m}}\frac{\gamma(d)}{\rho(\overline{m/d})},
\]
so the assertion follows by induction on $m$. Therefore, we obtain
\[
v_k(\eta^F(n))=v_k(\beta_i\gamma(m))\ge -1.
\]

Finally, suppose that
\[
N_{(p)}=\mathbb N_{(p)}.
\]
Then $\overline m=m$ for every $m\in\mathbb N_{(p)}$. Since $\rho$ is multiplicative,
\[
\sum_{d\mid m}\frac{\mu(d)}{\rho(d)}\frac{1}{\rho(m/d)}=\frac{1}{\rho(m)}\sum_{d\mid m}\mu(d)=\delta_{m,1}.
\]
By uniqueness of the Dirichlet inverse,
\[
\gamma(m)=\frac{\mu(m)}{\rho(m)}.
\]
\end{proof}

\begin{thm}\label{new-cong}
For every $m,n\in\mathbb N_{>0}$, the following assertions hold.
\begin{itemize}
\item[\textnormal{(i)}]
\[
v_k\bigl(\mathfrak F_{1,n}^F(T^m)\bigr)\geq v_q(m)-v_q(n)-1.
\]

\item[\textnormal{(ii)}]
\[
v_k\left(\sum_{d\mid(m,n)}\eta^F(d)\frac{\mathfrak F_{1,m/d}^F(T^{n/d})}{c_{\overline{n/d}}}\right)\geq -1.
\]

\item[\textnormal{(iii)}]
\begin{align*}
&v_k\left(\sum_{d\mid(m,n)}\eta^F(d)\frac{\mathfrak F_{1,m/d}^F(T^{n/d})}{c_{\overline{n/d}}}\right) \\
&\qquad\geq\min_{d\mid(m,n)}\left\{\left\lfloor\frac{1}{q-1}\left(\frac nd-1-\frac{mq}{d}\right)\right\rfloor_+-v_q\left(\frac nd\right)\right\}-1.
\end{align*}
\end{itemize}
\end{thm}

\begin{proof}
For any $n\in\mathbb{N}_{>0}$, since $\ded_k[V]/(V^n)$ is finite free over $\ded_k$, tensoring with $\ded_k[V]/(V^n)$ commutes with the finite equalizer defining the $F[\pi]$-invariants.
Hence
\[
\left(\ded_k[V]/(V^n)\otimes_{\ded_k}\mathcal R_k\right)^{F[\pi]}=\ded_k[V]/(V^n)\otimes_{\ded_k}\mathcal R_k^{F[\pi]}.
\]
Using
\[
\mathcal R_k^{F[\pi]}=\varphi_k(\mathcal R_k),
\]
we obtain
\[
\left(\ded_k[V]/(V^n)\otimes_{\ded_k}\mathcal R_k\right)^{F[\pi]}=(\id\otimes\varphi_k)\left(\ded_k[V]/(V^n)\otimes_{\ded_k}\mathcal R_k\right).
\]
For
\[
\overline{g}_n(V,T)\in\mathcal C^F(\ded_k[V]/(V^n)),
\]
the series
\[
\sum_{a\in F[\pi]}\hspace{-2mm}{{}^F}\bigl(\overline{g}_n(V,a+_FT)-_F\overline{g}_n(V,a)\bigr)
\]
is invariant under translation by $F[\pi]$.  It therefore belongs to the image of $\id\otimes\varphi_k$. Thus the Coleman $F$-norm extends uniquely to
\[
\mathcal C\text{-}N_{Co}^F:\mathcal C^F(\ded_k[V]/(V^n))\longrightarrow\mathcal C^F(\ded_k[V]/(V^n)).
\]
Passing to the inverse limit over $n$ gives
\[
\mathcal C\text{-}N_{Co}^F:\mathcal C^F(\ded_k[[V]])\longrightarrow\mathcal C^F(\ded_k[[V]]).
\]

Let $\widetilde{\Theta}^F$ be the coefficient adjustment operator, so that
\[
\widetilde{\Theta}^F:Tk[[V,T]]\longrightarrow Tk[[V,T]];~\sum_{n=1}^{\infty}a_n(V)T^n\longmapsto \sum_{n=1}^{\infty}c_{\overline n}a_n(V)T^n.
\]
Then
\begin{align*}
&\widetilde{\Theta}^F\log_F\mathcal C\text{-}N_{Co}^F\left(\sum_{r\in\mathcal R}\hspace{-1mm}{{}^F}(-_F)(-(VT)^r)\right)=\widetilde{\Theta}^F\left(\pi\widetilde{\text{Tr}}_{Co}\log_F\left(\sum_{r\in\mathcal R}\hspace{-1mm}{{}^F}(-_F)(-(VT)^r)\right)\right)\\
=&\pi\widetilde{\Theta}^F\widetilde{\text{Tr}}_{Co}\left(\sum_{r\in\mathcal R}-\log_F(-(VT)^r)\right)=\pi\widetilde{\Theta}^F\widetilde{\text{Tr}}_{Co}\left(\sum_{r\in\mathcal R}\sum_{n\in N}\frac{1}{c_n}(VT)^{rn}\right)\\
=&\pi\widetilde{\Theta}^F\widetilde{\text{Tr}}_{Co}\left(\sum_{m=1}^{\infty}\frac{1}{c_{\overline m}}(VT)^m\right)=\pi\sum_{m=1}^{\infty}\sum_{s=1}^{\infty}c_{\overline s}\frac{\mathfrak F_{1,s}^F(T^m)}{c_{\overline m}}V^mT^s\\
=&\sum_{m,s=1}^{\infty}c_{\overline s}\left(\pi\sum_{g\mid(m,s)}\frac{\mathfrak F_{1,s/g}^F(T^{m/g})}{c_{\overline{m/g}}}\delta_{g,1}\right)V^mT^s\\
=&\sum_{m,s=1}^{\infty}c_{\overline s}\left(\pi\sum_{g\mid(m,s)}\frac{\mathfrak F_{1,s/g}^F(T^{m/g})}{c_{\overline{m/g}}}\left(\sum_{d\mid g}\frac{\eta^F(d)}{c_{\overline{g/d}}}\right)\right)V^mT^s\\
=&\sum_{m,s=1}^{\infty}\left(\sum_{t\mid(m,s)}\frac{c_{\overline s}}{c_{\overline t}}\left(\pi\sum_{d\mid (m,s)/t}\eta^F(d)\frac{\mathfrak F_{1,s/(td)}^F(T^{m/(td)})}{c_{\overline{m/(td)}}}\right)\right)V^mT^s \\
=&\sum_{m,s=1}^{\infty}\left(\pi\sum_{d\mid(m,s)}\eta^F(d)\frac{\mathfrak F_{1,s/d}^F(T^{m/d})}{c_{\overline{m/d}}}\right)\left(\sum_{t=1}^{\infty}\frac{c_{\overline{st}}}{c_{\overline t}}(V^mT^s)^t\right)\\
=&\sum_{m,s=1}^{\infty}\left(\pi\sum_{d\mid(m,s)}\eta^F(d)\frac{\mathfrak F_{1,s/d}^F(T^{m/d})}{c_{\overline{m/d}}}\right)\widetilde{\Theta}^F\log_F\left(\sum_{r\in\mathcal R}\hspace{-1mm}{{}^F}(-_F)(-(V^mT^s)^r)\right) \\
=&\widetilde{\Theta}^F\log_F\left(\sum_{m,s=1}^{\infty}\hspace{-2mm}{{}^F}\mathcal C\text{-}\left[\pi\sum_{d\mid(m,s)}\eta^F(d)\frac{\mathfrak F_{1,s/d}^F(T^{m/d})}{c_{\overline{m/d}}}\right]^F\left(\sum_{r\in\mathcal R}\hspace{-1mm}{{}^F}(-_F)(-(V^mT^s)^r)\right)\right).
\end{align*}

For (i), since
\[
\pi\sum_{m=1}^{\infty}\sum_{s=1}^{\infty}c_{\overline s}\frac{\mathfrak F_{1,s}^F(T^m)}{c_{\overline m}}V^mT^s\in\ded_k[[V,T]],
\]
we have
\[
\pi c_{\overline s}\frac{\mathfrak F_{1,s}^F(T^m)}{c_{\overline m}}\in\ded_k.
\]
Consequently,
\[
v_k\bigl(\mathfrak F_{1,s}^F(T^m)\bigr)\geq v_k\left(\frac{c_{\overline m}}{\pi c_{\overline s}}\right)=v_q(m)-v_q(s)-1.
\]

For (ii), the identity above gives
\begin{align*}
&\mathcal C\text{-}N_{Co}^F\left(\sum_{r\in\mathcal R}\hspace{-1mm}{{}^F}(-_F)(-(VT)^r)\right) \\
&=\sum_{m,s=1}^{\infty}{}^F\mathcal C\text{-}\left[\pi\sum_{d\mid(m,s)}\eta^F(d)\frac{\mathfrak F_{1,s/d}^F(T^{m/d})}{c_{\overline{m/d}}}\right]^F\left(\sum_{r\in\mathcal R}\hspace{-1mm}{{}^F}(-_F)(-(V^mT^s)^r)\right).
\end{align*}
For $m,s\in\mathbb N_{>0}$, put
\[
\lambda_{m,s}=\pi\sum_{d\mid(m,s)}\eta^F(d)\frac{\mathfrak F_{1,s/d}^F(T^{m/d})}{c_{\overline{m/d}}}\in k.
\]
We prove that
\[
\lambda_{m,s}\in\ded_k
\]
by induction on the total degree $m+s$. Suppose that
\[
\lambda_{m',s'}\in\ded_k~~~(m'+s'<m+s).
\]
Then, for every such $(m',s')$,
\[
\mathcal C\text{-}[\lambda_{m',s'}]^F\left(\sum_{r\in\mathcal{R}}\hspace{-1mm}{{}^F}(-_F)(-(V^{m'}T^{s'})^r)\right)\in\ded_k[[V,T]].
\]
Hence all contributions to the coefficient of $V^mT^s$ coming from indices of smaller total degree are integral. On the other hand, an index $(m',s')$ with $m'+s'>m+s$ cannot contribute to the coefficient of $V^mT^s$. If $m'+s'=m+s$, then
\[
\mathcal C\text{-}[\lambda_{m',s'}]^F\left(\sum_{r\in\mathcal{R}}\hspace{-1mm}{{}^F}(-_F)(-(V^{m'}T^{s'})^r)\right)\equiv\lambda_{m',s'}V^{m'}T^{s'}\pmod{(V,T)^{m+s+1}}.
\]
Moreover, modulo $(V,T)^{m+s+1}$, the $F$-sum of terms of total degree $m+s$ agrees with their ordinary sum. Therefore the coefficient of $V^mT^s$ in the right-hand side differs from $\lambda_{m,s}$ by an element of $\ded_k$. Since the left-hand side belongs to $\ded_k[[V,T]]$, we obtain
\[
\lambda_{m,s}\in\ded_k.
\]
The induction is complete. Therefore,
\[
v_k\left(\sum_{d\mid(m,s)}\eta^F(d)\frac{\mathfrak F_{1,s/d}^F(T^{m/d})}{c_{\overline{m/d}}}\right)\ge -1.
\]
This proves (ii).

Finally, by the generalized Fleck--Sun--Wan congruence,
\begin{align*}
&v_k\left(\sum_{d\mid(m,n)}\eta^F(d)\frac{\mathfrak F_{1,m/d}^F(T^{n/d})}{c_{\overline{n/d}}}\right) \\
&\geq\min_{d\mid(m,n)}\left\{v_k(\eta^F(d))+v_k\bigl(\mathfrak F_{1,m/d}^F(T^{n/d})\bigr)-v_k(c_{\overline{n/d}})\right\} \\
&\geq\min_{d\mid(m,n)}\left\{\left\lfloor\frac{1}{q-1}\left(\frac nd-1-\frac{mq}{d}\right)\right\rfloor_+-v_q\left(\frac nd\right)\right\}-1.
\end{align*}
This proves (iii).
\end{proof}

\begin{proof}[Proof of Corollary~\ref{cor:intro-new-congruence}]
Take
\[
k=\mathbb Q_p,\qquad F=\widehat{\mathbb G}_m,\qquad q=p,
\]
with the coefficient-multiplicative logarithm
\[
\log_{\widehat{\mathbb G}_m}(T)=\sum_{n=1}^{\infty}\frac{(-1)^{n-1}}{n}T^n.
\]
Then
\[
N=\mathbb N_{>0},\qquad c_n=n,\qquad \overline n=n,
\]
and, by Lemma~\ref{lem:eta-F},
\[
\eta^F(d)=\frac{\mu(d)}{d}.
\]

Hence Theorem~\ref{new-cong}~(ii) gives
\[
v_p\left(\frac{1}{n}\sum_{d\mid(m,n)}\mu(d)\mathfrak F_{1,m/d}^{\widehat{\mathbb G}_m}(T^{n/d})\right)\ge -1.
\]
Therefore
\[
v_p\left(\sum_{d\mid(m,n)}\mu(d)\mathfrak F_{1,m/d}^{\widehat{\mathbb G}_m}(T^{n/d})\right)\ge v_p(n)-1.
\]

By the preceding corollary,
\[
\mathfrak F_{1,m/d}^{\widehat{\mathbb G}_m}(T^{n/d})=(-1)^{n/d}\sum_{\substack{0\le a\le n/d\\ p\mid a}}(-1)^a\binom{n/d}{a}\binom{a/p}{m/d}.
\]
Substituting this identity yields
\[
v_p\left(\sum_{d\mid(m,n)}\mu(d)(-1)^{n/d}\sum_{\substack{0\le a\le n/d\\p\mid a}}(-1)^a\binom{n/d}{a}\binom{a/p}{m/d}\right)\ge v_p(n)-1.
\]

If $v_p(n)\ge 2$, the asserted congruence follows immediately.
\end{proof}

\begin{thm}
Define
\begin{align*}
N_{Co}^F=\sum_{m,n=1}^{\infty}{}^S\left[\pi\sum_{d\mid(m,n)}\eta^F(d)\frac{\mathfrak F_{1,m/d}^F(T^{n/d})}{c_{\overline{n/d}}}\right]_F^{\id}(v_m^F\circ f_n^F):W^F(\ded_k)\longrightarrow W^F(\ded_k),
\end{align*}
and
\[
\operatorname{ghost}\text{-}N_{Co}^F:\operatorname{Ghost}(\ded_k)\longrightarrow\operatorname{Ghost}(\ded_k)
\]
by
\[
\{x_n\}\longmapsto\left\{\pi c_{\overline n}\sum_{m=1}^{\infty}\frac{\mathfrak F_{1,n}^F(T^m)}{c_{\overline m}}x_m\right\}_n.
\]
Then the diagrams
\[
\xymatrix{
W^F(\ded_k) \ar[d]_{N_{Co}^F} \ar[r]^{w^F} & \operatorname{Ghost}(\ded_k) \ar[d]^{\operatorname{ghost}\text{-}N_{Co}^F} \\
W^F(\ded_k) \ar[r]_{w^F} & \operatorname{Ghost}(\ded_k) \ar@{}[lu]|{\circlearrowright}
}
\]
and
\[
\xymatrix{
W^F(\ded_k) \ar[r]^{N_{Co}^F} \ar[d]_{e^F} & W^F(\ded_k) \ar[d]^{e^F} \\
\mathcal C^F(\ded_k) \ar[r]_{\mathcal C\text{-}N_{Co}^F} & \mathcal C^F(\ded_k) \ar@{}[lu]|{\circlearrowright}
}
\]
are commutative.
\end{thm}

\begin{proof}
Define
\[
\widetilde{\Theta}^F:Tk[[T]]\longrightarrow Tk[[T]];~\sum_{n=1}^{\infty}a_nT^n\longmapsto\sum_{n=1}^{\infty}c_{\overline n}a_nT^n.
\]
Let $g\in\mathcal C^F(\ded_k)$, and write $\widetilde{\ell}^F(g)=\{x_n\}_{n\ge1}$. Equivalently, $\ds \widetilde{\Theta}^F\log_F g(T)=\sum_{n=1}^{\infty}x_nT^n$, and hence
\[
\log_F g(T)=\sum_{n=1}^{\infty}\frac{x_n}{c_{\overline n}}T^n.
\]
Therefore
\begin{align*}
&\widetilde{\Theta}^F\log_F\bigl(\mathcal C\text{-}N_{Co}^F(g)\bigr)\\
&=\widetilde{\Theta}^F\left(\pi\widetilde{\operatorname{Tr}}_{Co}\left(\sum_{n=1}^{\infty}\frac{x_n}{c_{\overline n}}T^n\right)\right)\\
&=\sum_{m=1}^{\infty}\left(\pi c_{\overline m}\sum_{n=1}^{\infty}\frac{x_n}{c_{\overline n}}\mathfrak F_{1,m}^F(T^n)\right)T^m.
\end{align*}
Thus
\[
\widetilde{\ell}^F(\mathcal{C}\text{-}N_{Co}^F(g))=\operatorname{ghost}\text{-}N_{Co}^F(\widetilde{\ell}^F(g)).
\]

Next, by Theorem~\ref{new-cong} (iii), the Cartier series converges coordinatewise $\pi$-adically. Let $\{v_n\}\in W^F(\ded_k)$, then
\begin{align*}
&(w^F\circ N_{Co}^F)(\{v_n\}) \\
&=w^F\left(\sum_{m,n=1}^{\infty}{}^S\left[\pi\sum_{d\mid(m,n)}\eta^F(d)\frac{\mathfrak F_{1,m/d}^F(T^{n/d})}{c_{\overline{n/d}}}\right]_F^{\id}(v_m^F\circ f_n^F)(\{v_n\})\right) \\
&=\sum_{m,n=1}^{\infty}\pi\sum_{d\mid(m,n)}\eta^F(d)\frac{\mathfrak F_{1,m/d}^F(T^{n/d})}{c_{\overline{n/d}}}(w^F\circ v_m^F\circ f_n^F)(\{v_n\}).
\end{align*}
Using the ghost descriptions of Verschiebung and Frobenius, this is
\begin{align*}
&\left\{\sum_{m\mid r}\sum_{n=1}^{\infty}\pi\sum_{d\mid(m,n)}\eta^F(d)\frac{\mathfrak F_{1,m/d}^F(T^{n/d})}{c_{\overline{n/d}}}\frac{c_{\overline r}}{c_{\overline{r/m}}}W_{\overline{nr/m}}^F\bigl(v_{r_{nr/m}j}\mid j|_N\overline{nr/m}\bigr)\right\}_r.
\end{align*}
Reindexing the sum gives
\begin{align*}
&\left\{c_{\overline r}\sum_{s=1}^{\infty}\sum_{t\mid(r,s)}\pi\frac{\mathfrak F_{1,rt/(r,s)}^F\left(T^{st/(r,s)}\right)}{c_{\overline{st/(r,s)}}}\left(\sum_{d\mid (r,s)/t}\frac{\eta^F(d)}{c_{\overline{(r,s)/(td)}}}\right)\right. \\
&\hspace{7cm}\left.\times W_{\overline s}^F(v_{r_sj}\mid j|_N\overline s)\right\}_r.
\end{align*}
By the defining relation for $\eta^F$, this becomes
\begin{align*}
&\left\{c_{\overline r}\sum_{s=1}^{\infty}\pi\frac{\mathfrak F_{1,r}^F(T^s)}{c_{\overline s}}W_{\overline s}^F(v_{r_sj}\mid j|_N\overline s)\right\}_r \\
&=(\operatorname{ghost}\text{-}N_{Co}^F\circ w^F)(\{v_n\}).
\end{align*}
Hence the first diagram is commutative. The second follows from the curve interpretation $e^F$.
\end{proof}

\begin{thm}
Let $B$ be a $\pi$-adically complete Hausdorff $\ded_k$-algebra.
Then the maps
\begin{align*}
\mathcal C\text{-}N_{Co}^F:\mathcal C^F(B)&\longrightarrow\mathcal C^F(B);\\
g(T)&\longmapsto\sum_{m,n=1}^{\infty}\hspace{-2mm}{{}^F}\mathcal C\text{-}\left[\pi\sum_{d\mid(m,n)}\eta^F(d)\frac{\mathfrak F_{1,m/d}^F(T^{n/d})}{c_{\overline{n/d}}}\right]_F^{\id}(\mathcal C\text{-}v_m^F\circ\mathcal C\text{-}f_n^F)(g),
\end{align*}
\begin{align*}
N_{Co}^F=\sum_{m,n=1}^{\infty}\hspace{-2mm}{{}^S}\left[\pi\sum_{d\mid(m,n)}\eta^F(d)\frac{\mathfrak F_{1,m/d}^F(T^{n/d})}{c_{\overline{n/d}}}\right]_F^{\id}(v_m^F\circ f_n^F):W^F(B)\longrightarrow W^F(B),
\end{align*}
and
\[
\operatorname{ghost}\text{-}N_{Co}^F:\operatorname{Ghost}(B)\longrightarrow\operatorname{Ghost}(B);
\]
\[
\{x_n\}\longmapsto\left\{\pi c_{\overline n}\sum_{m=1}^{\infty}\frac{\mathfrak F_{1,n}^F(T^m)}{c_{\overline m}}x_m\right\}_n
\]
are well-defined, and the diagrams
\[
\xymatrix{
W^F(B) \ar[d]_{N_{Co}^F} \ar[r]^{w^F} & \operatorname{Ghost}(B) \ar[d]^{\operatorname{ghost}\text{-}N_{Co}^F} \\
W^F(B) \ar[r]_{w^F} & \operatorname{Ghost}(B) \ar@{}[lu]|{\circlearrowright}
}
\]
and
\[
\xymatrix{
W^F(B) \ar[r]^{N_{Co}^F} \ar[d]_{e^F} & W^F(B) \ar[d]^{e^F} \\
\mathcal C^F(B) \ar[r]_{\mathcal C\text{-}N_{Co}^F} & \mathcal C^F(B) \ar@{}[lu]|{\circlearrowright}
}
\]
are commutative.
\end{thm}

\begin{proof}
For each fixed $m\in\mathbb N_{>0}$, Theorem~\ref{new-cong}~(iii) gives
\begin{align*}
&v_k\left(\pi\sum_{d\mid(m,n)}\eta^F(d)\frac{\mathfrak F_{1,m/d}^F(T^{n/d})}{c_{\overline{n/d}}}\right) \\
&\qquad\ge\min_{d\mid(m,n)}\left\{\left\lfloor\frac{1}{q-1}\left(\frac nd-1-\frac{mq}{d}\right)\right\rfloor_+-v_q\left(\frac nd\right)\right\}.
\end{align*}
Since $d$ ranges over the finitely many divisors of the fixed integer $m$, the right-hand side tends to infinity as $n\to\infty$. Hence
\[
\pi\sum_{d\mid(m,n)}\eta^F(d)\frac{\mathfrak F_{1,m/d}^F(T^{n/d})}{c_{\overline{n/d}}}\longrightarrow0
\]
$\pi$-adically as $n\to\infty$. Moreover, for a fixed output Witt coordinate $r$, the Verschiebung $v_m^F$ can contribute only when $m\mid r$. Thus only finitely many values of $m$ occur in the $r$-th coordinate, and the remaining $n$-series is $\pi$-adically convergent. Since $B$ is complete and Hausdorff, the Cartier expansion therefore defines a continuous map
\[
N_{Co}^F:W^F(B)\longrightarrow W^F(B).
\]
On the other hand, for each fixed $r$, Theorem~\ref{general-FSW}~(i) gives
\[
v_k\left(\pi c_{\overline r}\frac{\mathfrak F_{1,r}^F(T^s)}{c_{\overline s}}\right)\ge1+v_q(r)-v_q(s)+\left\lfloor\frac{s-1-rq}{q-1}\right\rfloor_+,
\]
which tends to infinity as $s\to\infty$. Hence
\[
\sum_{s=1}^{\infty}\pi c_{\overline r}\frac{\mathfrak F_{1,r}^F(T^s)}{c_{\overline s}}x_s
\]
converges $\pi$-adically in $B$, and therefore $\operatorname{ghost}\text{-}N_{Co}^F$ is well-defined.

The curve-side operator is obtained by transport through the isomorphism
\[
e^F:W^F(B)\xrightarrow{\sim}\mathcal C^F(B).
\]
For every finite partial sum, the compatibility of $e^F$ with Verschiebung, Frobenius, and scalar operations gives, for any $M\in\mathbb{N}_{>0}$,
\begin{align*}
&e^F\circ \left(\sum_{m,n=1}^M\hspace{-2mm}{{}^{S}}\left[\pi\sum_{d\mid(m,n)}\eta^F(d)\frac{\mathfrak F_{1,m/d}^F(T^{n/d})}{c_{\overline{n/d}}}\right]_F^{\id}(v_m^F\circ f_n^F)\right) \\
=&\left(\sum_{m,n=1}^M\hspace{-2mm}{{}^{F}}\mathcal C\text{-}\left[\pi\sum_{d\mid(m,n)}\eta^F(d)\frac{\mathfrak F_{1,m/d}^F(T^{n/d})}{c_{\overline{n/d}}}\right]_F^{\id}(\mathcal C\text{-}v_m^F\circ\mathcal C\text{-}f_n^F)\right)\circ e^F.
\end{align*}
Since $e^F$ and $(e^F)^{-1}$ are continuous for the $\pi$-adic coordinatewise topology, passing to the limit $M$ shows that the displayed curve-side series converges and agrees with the operator transported from $N_{Co}^F$.

Similarly, the compatibility of the ghost map with Verschiebung, Frobenius, and scalar operations gives the first commutative diagram for finite partial sums, and passing to the $\pi$-adic limit gives the asserted diagram.
\end{proof}

\begin{cor}
The Coleman $F_q$-norm preserves
\[
W_{\langle q\rangle}^{F_q}(\kappa_k).
\]
Under the canonical $\ded_k$-algebra isomorphism
\[
\theta:W_{\langle q\rangle}^{F_q}(\kappa_k)\xrightarrow{\sim}\ded_k,
\]
it is identified with multiplication by $q/\pi$. In particular, the diagram
\[
\xymatrix{
W_{\langle q\rangle}^{F_q}(\kappa_k) \ar[r]^{N_{Co}^{F_q}} \ar[d]_{\theta}^{\rotatebox{90}{$\sim$}} & W_{\langle q\rangle}^{F_q}(\kappa_k) \ar[d]^{\theta}_{\rotatebox{90}{$\sim$}} \\
\ded_k \ar[r]_{q/\pi} & \ded_k \ar@{}[lu]|{\circlearrowright}
}
\]
is commutative.
\end{cor}

\begin{proof}
We first compute the image of the multiplicative identity
\[
\mathbf{1}\in W_{\langle q\rangle}^{F_q}(\ded_k),
\]
regarded as an element of the $r=1$ factor of $W^{F_q}(\ded_k)$.

Since $q$ is odd, all exponents occurring in $\log_{F_q}(T)$ are odd, and hence
\[
(-_{F_q})(-T)=T.
\]
Therefore
\[
e^{F_q}(\mathbf{1})=T.
\]

Let
\[
h=\mathcal C\text{-}N_{Co}^{F_q}(T).
\]
By the definition of the Coleman $F_q$-norm on curves,
\begin{align*}
h([\pi]_{F_q}(T))=\sum_{a\in F_q[\pi]}\hspace{-3mm}{{}^{F_q}}\bigl((a+_{F_q}T)-_{F_q}a\bigr)=\sum_{a\in F_q[\pi]}\hspace{-3mm}{{}^{F_q}}T=[q]_{F_q}(T).
\end{align*}
Since
\[
[q]_{F_q}=\left[\frac{q}{\pi}\right]_{F_q}\circ[\pi]_{F_q},
\]
we obtain
\[
h([\pi]_{F_q}(T))=\left[\frac{q}{\pi}\right]_{F_q}([\pi]_{F_q}(T)).
\]
Since substitution by $[\pi]_{F_q}(T)$ is injective on $\ded_k[[T]]$, we obtain
\[
\mathcal C\text{-}N_{Co}^{F_q}(T)=\left[\frac{q}{\pi}\right]_{F_q}(T).
\]
Consequently,
\[
N_{Co}^{F_q}(\mathbf{1})=\left[\frac{q}{\pi}\right]^{F_q}(\mathbf{1}).
\]

We next show that $N_{Co}^{F_q}$ commutes with the $\ded_k$-scalar action. For every $\alpha\in\ded_k$ and $g\in\mathcal C^{F_q}(\ded_k)$, we have
\begin{align*}
&\bigl(\mathcal C\text{-}N_{Co}^{F_q}(\mathcal C\text{-}[\alpha]^{F_q}(g))\bigr)([\pi]_{F_q}(T)) \\
&=\sum_{a\in F_q[\pi]}\hspace{-3mm}{{}^{F_q}}\left([\alpha]_{F_q}(g(a+_{F_q}T))-_{F_q}[\alpha]_{F_q}(g(a))\right) \\
&=[\alpha]_{F_q}\left(\sum_{a\in F_q[\pi]}\hspace{-3mm}{{}^{F_q}}\bigl(g(a+_{F_q}T)-_{F_q}g(a)\bigr)\right) \\
&=[\alpha]_{F_q}\left((\mathcal C\text{-}N_{Co}^{F_q}(g))([\pi]_{F_q}(T))\right).
\end{align*}
Again, since substitution by $[\pi]_{F_q}(T)$ is injective, we obtain
\[
\mathcal C\text{-}N_{Co}^{F_q}\circ\mathcal C\text{-}[\alpha]^{F_q}=\mathcal C\text{-}[\alpha]^{F_q}\circ\mathcal C\text{-}N_{Co}^{F_q},
\]
and hence
\[
N_{Co}^{F_q}\circ[\alpha]^{F_q}=[\alpha]^{F_q}\circ N_{Co}^{F_q}
\]
on $W^{F_q}(\ded_k)$.

Let
\[
\psi:\ded_k\longrightarrow\kappa_k
\]
be the natural projection, and put
\[
\overline{\mathbf{1}}=W_{\langle q\rangle}^{F_q}(\psi)(\mathbf{1}).
\]
By the functoriality of the Coleman $F_q$-norm, the diagram
\[
\xymatrix{
W^{F_q}(\ded_k) \ar[r]^{N_{Co}^{F_q}} \ar[d]_{W^{F_q}(\psi)} & W^{F_q}(\ded_k) \ar[d]^{W^{F_q}(\psi)} \\
W^{F_q}(\kappa_k) \ar[r]_{N_{Co}^{F_q}} & W^{F_q}(\kappa_k) \ar@{}[lu]|{\circlearrowright}
}
\]
is commutative. Hence
\[
N_{Co}^{F_q}(\overline{\mathbf{1}})=\left[\frac{q}{\pi}\right]^{F_q}(\overline{\mathbf{1}}).
\]

Moreover, since the scalar action is also natural with respect to $\ded_k$-algebra homomorphisms, the relation
\[
N_{Co}^{F_q}\circ[\alpha]^{F_q}=[\alpha]^{F_q}\circ N_{Co}^{F_q}
\]
descends to $W^{F_q}(\kappa_k)$ for every $\alpha\in\ded_k$.

Now let
\[
x\in W_{\langle q\rangle}^{F_q}(\kappa_k).
\]
Since
\[
\theta:W_{\langle q\rangle}^{F_q}(\kappa_k)\xrightarrow{\sim}\ded_k
\]
is an isomorphism of $\ded_k$-algebras, there exists $\alpha\in\ded_k$ such that
\[
x=[\alpha]^{F_q}(\overline{\mathbf{1}}).
\]
Therefore
\begin{align*}
N_{Co}^{F_q}(x)&=N_{Co}^{F_q}\left([\alpha]^{F_q}(\overline{\mathbf{1}})\right)\\
&=[\alpha]^{F_q}\left(N_{Co}^{F_q}(\overline{\mathbf{1}})\right) \\
&=[\alpha]^{F_q}\left(\left[\frac{q}{\pi}\right]^{F_q}(\overline{\mathbf{1}})\right).
\end{align*}
In particular,
\[
N_{Co}^{F_q}(x)\in W_{\langle q\rangle}^{F_q}(\kappa_k),
\]
so the $q$-typical factor is preserved. Applying $\theta$, we obtain
\[
\theta\bigl(N_{Co}^{F_q}(x)\bigr)=\alpha\frac{q}{\pi}=\frac{q}{\pi}\theta(x).
\]
This proves the assertion.
\end{proof}

\end{document}